\documentclass{amsart}

\usepackage{amsmath, amssymb}
\usepackage{bbm}

\newtheorem{theorem}{Theorem}[section]
\newtheorem{lemma}[theorem]{Lemma}
\newtheorem{proposition}[theorem]{Proposition}
\newtheorem{corollary}[theorem]{Corollary}
\newtheorem{claim}[theorem]{Claim}

\newtheorem{fact}[theorem]{Fact}

\newenvironment{definition}[1][Definition]{\begin{trivlist}
\item[\hskip \labelsep {\bfseries #1}]}{\end{trivlist}}

\newenvironment{remark}[1][Remark]{\begin{trivlist}
\item[\hskip \labelsep {\bfseries #1}]}{\end{trivlist}}
\newenvironment{question}[1][Question]{\begin{trivlist}
\item[\hskip \labelsep {\bfseries #1}]}{\end{trivlist}}

\newcommand{\cf}{\mathrm{cf}}

\newcommand{\bb}{\mathbb}

\begin{document}
\title{Robust reflection principles}
\author{Chris Lambie-Hanson}
\address{Einstein Institute of Mathematics, Hebrew University of Jerusalem \\ Jerusalem, 91904, Israel}
\email{clambiehanson@math.huji.ac.il}
\thanks{This research was undertaken while the author was a Lady Davis Postdoctoral Fellow. The author would like to thank the Lady Davis Fellowship Trust and the Hebrew University of Jerusalem. The author would also like to thank Menachem Magidor for many helpful discussions.}
\begin{abstract}
	A cardinal $\lambda$ satisfies a property P \emph{robustly} if, whenever $\bb{Q}$ is a forcing poset and $|\bb{Q}|^+ < \lambda$, $\lambda$ satisfies P in $V^{\bb{Q}}$. We study the extent to which certain reflection properties of large cardinals can be satisfied robustly by small cardinals. We focus in particular on stationary reflection and the tree property, both of which can consistently hold but fail to be robust at small cardinals. We introduce natural strengthenings of these principles which are always robust and which hold at sufficiently large cardinals, consider the extent to which these strengthenings are in fact stronger than the original principles, and investigate the possibility of these strengthenings holding at small cardinals, particularly at successors of singular cardinals. 
\end{abstract}
\maketitle

\section{Introduction}

Large cardinal properties have, among others, the following two appealing attributes: they imply certain strong reflection properties, and they are robust under small forcing. The study of the extent to which reflection properties of large cardinals can hold at small cardinals, and in particular at successors of singular cardinals, has been a fruitful line of research in set theory. We continue this line here, adding the requirement of robustness under small forcing to these reflection properties and focusing in particular on stationary reflection and the tree property at successors of singular cardinals. 

\begin{definition}
	Let $P$ be a property that can hold of a cardinal $\lambda$. We say that $\lambda$ \emph{satisfies $P$ robustly} or that $\lambda$ \emph{has the robust property $P$} if, whenever $\bb{Q}$ is a forcing poset and $|\bb{Q}|^+ < \lambda$, $\lambda$ satisfies $P$ in $V^{\bb{Q}}$.
\end{definition}

\begin{remark}
	The requirement $|\bb{Q}|^+ < \lambda$, rather than the seemingly more natural $|\bb{Q}| < \lambda$, is necessary for our purposes in order to obtain consistent principles. If $\lambda = \mu^+$ and $\bb{Q} = \mathrm{Coll}(\omega, \mu)$, then $|\bb{Q}| < \lambda$ and, in $V^{\bb{Q}}$, $\lambda = \omega_1$ and therefore cannot satisfy, for example, stationary reflection or the tree property. 
\end{remark}

Most large cardinal properties are always robust. For example, if $\lambda$ satisfies the property ``is inaccessible," ``is weakly compact,"  ``is measurable," ``is strongly compact," ``is supercompact," etc., then, by an argument of Levy and Solovay (see \cite{levysolovay}), $\lambda$ satisfies the property robustly. Therefore, reflection principles, when they hold due to large cardinal properties, are themselves robust. Of particular interest to us are the following.

\begin{fact}
	Suppose $\lambda$ is weakly compact. Then $\lambda$ satisfies robust stationary reflection and the robust tree property.
\end{fact}

\begin{fact}
	Suppose $\mu$ is a singular limit of strongly compact cardinals and $\lambda = \mu^+$. Then $\lambda$ satisfies robust stationary reflection and the robust tree property.
\end{fact}

As we will see, though, these reflection principles need not be robust when they hold at small cardinals. We consider here natural strengthenings of reflection principles, in particular stationary reflection and the tree property, that are always robust, and investigate the extent to which they can hold at small cardinals and the extent to which they are true strengthenings of the more classical principles.

The general outline of the paper is as follows. In Section \ref{reflSect}, we consider robust stationary reflection. We show that this is equivalent to a natural condition studied by Cummings and the author in \cite{reflection} and that it is not in general equivalent to stationary reflection at inaccessible cardinals. The rest of the paper is devoted to the tree property and the strong system property. In Section \ref{systemSect}, we introduce the strong system property, a robust strengthening of the tree property that is equivalent to the tree property at inaccessible cardinals. In Section \ref{preservationSect}, we present some strengthenings of branch preservation lemmas for systems, due in their original form to Sinapova (see \cite{sinapova}). In Section \ref{collapsingSect}, we prove a technical lemma about systems in a generic extension by a product of Levy collapses. In Section \ref{weakSquareSect}, we show that fairly weak square principles imply the failure of the strong system property, and we provide a characterization of the robustness of having no special $\mu^+$-Aronszajn trees for infinite $\mu$. In Section \ref{downsection}, we adapt arguments of Fontanella and Magidor from \cite{fontanella} to show that the strong system property can consistently hold at $\aleph_{\omega^2 + 1}$. In Section \ref{separatingSect}, we show that we have some control over the extent of the failure of the strong system property at $\aleph_{\omega^2 + 1}$ and conclude with some open questions.

Our notation is, for the most part, standard. \cite{jech} is our standard reference for all undefined terms and notations. If $\kappa < \lambda$ are infinite cardinals, with $\kappa$ regular, then $S^\lambda_{\kappa} = \{\alpha < \lambda \mid \cf(\alpha) = \kappa \}$. $S^\lambda_{<\kappa}$, $S^\lambda_{\leq \kappa}$, etc. are given the natural meanings. By `inaccessible,' we always mean `strongly inaccessible.' If $R$ is a binary relation on a set $X$, we will often write $x <_R y$ in place of $(x,y) \in R$.

\section{Stationary reflection} \label{reflSect}

Recall the following definitions.

\begin{definition}
	Let $\lambda$ be a regular, uncountable cardinal, and let $S \subseteq \lambda$ be stationary.
	\begin{enumerate}
		\item{If $\alpha < \lambda$ and $\cf(\alpha) > \omega$, then $S$ \emph{reflects} at $\alpha$ if $S \cap \alpha$ is stationary in $\alpha$.}
		\item{$S$ \emph{reflects} if there is $\alpha < \lambda$ such that $S$ reflects at $\alpha$.}
		\item{$S$ \emph{reflects at arbitrarily high cofinalities} if, for every regular $\kappa < \lambda$, there is $\alpha < \lambda$ such that $\cf(\alpha) \geq \kappa$ and $S$ reflects at $\alpha$.}
		\item{\emph{$\mathrm{Refl}(\lambda)$} is the assertion that every stationary subset of $\lambda$ reflects.}
	\end{enumerate}
\end{definition}

The following proposition is easily proven (see \cite{reflection}, for example).

\begin{proposition}
	Suppose $\mathrm{Refl}(\aleph_{\omega + 1})$ holds. Then every stationary subset of $\aleph_{\omega + 1}$ reflects at arbitrarily high cofinalities.
\end{proposition}

Also, standard arguments yield that, if $\lambda$ is weakly compact, then every stationary subset of $\lambda$ reflects at arbitrarily high cofinalities. However, the situation is different in general. In \cite{reflection}, Cummings and the author show that, assuming the existence of sufficiently large cardinals, it is consistent that there is a singular cardinal $\mu > \aleph_\omega$ such that $\mathrm{Refl}(\mu^+)$ holds and there is a stationary subset of $\mu^+$ that does not reflect at arbitrarily high cofinalities. In \cite{reflection2}, the author extends this result to show that, assuming the existence of a proper class of supercompact cardinals, there is a class forcing extension in which, whenever $\mu > \aleph_\omega$ is a singular cardinal, $\mathrm{Refl}(\mu^+)$ holds and there is a stationary subset of $\mu^+$ that does not reflect at arbitrarily high cofinalities.

It turns out that this notion is closely related to robust stationary reflection.

\begin{theorem}
	Suppose $\lambda$ is a regular, uncountable cardinal. The following are equivalent.
	\begin{enumerate}
		\item{$\lambda$ satisfies robust stationary reflection.}
		\item{Every stationary subset of $\lambda$ reflects at arbitrarily high cofinalities.}
	\end{enumerate}
\end{theorem}

\begin{proof}
	First note that, if $\lambda = \kappa^+$ and $\kappa$ is regular, then $S^{\lambda}_\kappa$ is a non-reflecting stationary subset of $\lambda$. Hence, both (1) and (2) imply that $\lambda$ is either weakly inaccessible or the successor of a singular cardinal. In particular, if $\kappa < \lambda$ is regular, then $\kappa^+ < \lambda$.

	(1) $\Rightarrow$ (2): Assume (1) holds. Suppose for sake of contradiction that $S \subseteq \lambda$ is stationary, $\kappa < \lambda$ is regular, and $S$ does not reflect at any ordinal in $S^{\lambda}_{\geq \kappa}$. Let $\bb{P} = \mathrm{Coll}(\omega, \kappa)$. $|\bb{P}| = \kappa$, so $|\bb{P}|^+ < \lambda$. In particular, $\bb{P}$ has the $\lambda$-c.c., so $S$ remains stationary in $V^{\bb{P}}$. Also, if $\alpha < \lambda$ and $\cf^{V^{\bb{P}}}(\alpha) > \omega$, then $\cf^V(\alpha) > \kappa$. Since $S$ does not reflect at any ordinal in $S^{\lambda}_{\geq \kappa}$ in $V$, there is a club $C_\alpha$ in $\alpha$ such that $C_\alpha \cap S = \emptyset$. $C_\alpha$ still witnesses that $S$ does not reflect at $\alpha$ in $V^{\bb{P}}$, so $S$ is a non-reflecting stationary subset of $\lambda$ in $V^{\bb{P}}$, contradicting (1).

	(2) $\Rightarrow$ (1): Assume (2) holds. Suppose for sake of contradiction that $|\bb{P}|$ is a forcing poset, $|\bb{P}|^+ < \lambda$, $p \in \bb{P}$, and $\dot{S}$ is a $\bb{P}$-name such that $p \Vdash ``\dot{S}$ is a non-reflecting stationary subset of $\lambda$." For all $q \leq p$, let $S_q = \{\eta < \lambda \mid q \Vdash ``\eta \in \dot{S}" \}$. Since $p \Vdash ``\dot{S} \subseteq \bigcup_{q \leq p} S_q"$ and $|\bb{P}| < \lambda$, there must be $q \leq p$ such that $S_q$ is stationary in $\lambda$. Fix such a $q$. By (2), we may find $\alpha \in S^\lambda_{\geq |\bb{P}|^+}$ such that $S_q$ reflects at $\alpha$. Since $\bb{P}$ trivially has the $|\bb{P}|^+$-c.c., $S_q \cap \alpha$ remains stationary in $V^{\bb{P}}$. But then, since $q \Vdash ``S_q \subseteq \dot{S}"$, we have $q \Vdash ``\dot{S}$ reflects at $\alpha,"$ which is a contradiction.
\end{proof}

The next result shows that robust stationary reflection is not necessarily equivalent to stationary reflection, even for inaccessible cardinals.

\begin{theorem}
	Suppose there is an inaccessible limit of supercompact cardinals. Then there is a forcing extension with an inaccessible cardinal $\lambda$ such that $\mathrm{Refl}(\lambda)$ holds but there is a stationary $S \subseteq S^\lambda_\omega$ that does not reflect at any ordinal in $S^\lambda_{>\aleph_\omega}$.
\end{theorem}

\begin{proof}
	The proof largely follows the proof of Theorem 3.1 in \cite{reflection}, so we omit many of the details. Let $\lambda$ be the least inaccessible limit of supercompact cardinals. In particular, $\lambda$ is not Mahlo. Let $\langle \kappa_i \mid i < \lambda \rangle$ be an increasing, continuous sequence of cardinals such that:
	\begin{itemize}
		\item{$\kappa_0 = \omega$;}
		\item{if $i = 0$ or $i$ is a limit ordinal, $\kappa_{i+1} = \kappa_i^+$;}
		\item{if $i$ is a successor ordinal, $\kappa_{i+1}$ is supercompact;}
		\item{$\sup(\{\kappa_i \mid i < \lambda\}) = \lambda$.}
	\end{itemize}
	We first define a forcing iteration $\langle \bb{P}_i, \dot{\bb{Q}}_j \mid i \leq \lambda, j < \lambda \rangle$, taken with full supports, as follows. If $i=0$ or $i$ is a limit ordinal, then $\Vdash_{\bb{P}_i} ``\dot{\bb{Q}}_i$ is trivial." If $i$ is a successor ordinal, then $\Vdash_{\bb{P}_i}``\dot{\bb{Q}}_i = \mathrm{Coll}(\kappa_i, < \kappa_{i+1})."$ Let $\bb{P} = \bb{P}_\lambda$. In $V^{\bb{P}}$, $\lambda$ is the least inaccessible cardinal and, for all $i < \lambda$, $\kappa_i = \aleph_i$. In $V^{\bb{P}}$, let $\vec{a} = \langle a_\alpha \mid \alpha < \lambda \rangle$ be an enumeration of all bounded subsets of $\lambda$, and let $\bb{A}$ be the forcing to shoot a club through the set of ordinals below $\lambda$ that are approachable with respect to $\vec{a}$. In $V^{\bb{P} * \dot{\bb{A}}}$, let $\bb{S}$ be the poset whose conditions are of the form $s = (\gamma^s, x^s)$ such that:
	\begin{itemize}
		\item{$\gamma^s < \lambda$;}
		\item{$x^s \subseteq (\gamma^s + 1) \cap \mathrm{cof}(\omega)$;}
		\item{for all $\beta \in S^\lambda_{\geq \aleph_\omega}$, $x^s \cap \beta$ is not stationary in $\beta$.}
	\end{itemize}
	If $s,t \in \bb{S}$, then $t \leq s$ iff $\gamma^t \geq \gamma^s$ and $x^t \cap (\gamma^s + 1) = x^s$.

	If $G*H*I$ is generic for $\bb{P} * \dot{\bb{A}} * \dot{\bb{S}}$, then $V[G*H*I]$ is the desired model, with $S = \bigcup_{s \in I} x^s$ being the witnessing stationary subset of $S^\lambda_\omega$ not reflecting at any ordinal in $S^\lambda_{>\aleph_\omega}$. The verification is as in \cite{reflection} and is thus omitted.
\end{proof}

\section{Systems} \label{systemSect}

\begin{definition}
	Let $R$ be a binary relation on a set $X$.
	\begin{itemize}
		\item{If $a,b \in X$, then $a$ and $b$ are \emph{$R$-comparable} if $a = b$, $a <_R b$, or $b <_R a$. Otherwise, $a$ and $b$ are \emph{$R$-incomparable}, which is denoted $a \perp_R b$.}
		\item{$R$ is \emph{tree-like} if, for all $a,b,c \in X$, if $a <_R c$ and $b <_R c$, then $a$ and $b$ are $R$-comparable.}
	\end{itemize}
\end{definition}

\begin{definition}
	Let $\lambda$ be an infinite, regular cardinal. $S = \langle \{\{\alpha\} \times \kappa_\alpha \mid \alpha \in I\}, \mathcal{R} \rangle$ is a \emph{$\lambda$-system} if:
	\begin{enumerate}
		\item{$I \subseteq \lambda$ is unbounded and, for all $\alpha \in I$, $0 < \kappa_\alpha < \lambda$. We sometimes identify $S$ with $\{\{\alpha\} \times \kappa_\alpha \mid \alpha \in I \}$. For each $\alpha \in I$, we say that $S_\alpha := \{\alpha\} \times \kappa_\alpha$ is the $\alpha^{\mathrm{th}}$ level of $S$;}
		\item{$\mathcal{R}$ is a set of binary, transitive, tree-like relations on $S$ and $|\mathcal{R}| < \lambda$;}
		\item{for all $R \in \mathcal{R}$, $\alpha_0, \alpha_1 \in I$, $\beta_0 < \kappa_{\alpha_0}$, and $\beta_1 < \kappa_{\alpha_1}$, if $(\alpha_0, \beta_0) <_R (\alpha_1,\beta_1)$, then $\alpha_0 < \alpha_1$;}
		\item{for all $\alpha_0 < \alpha_1$, both in $I$, there are $\beta_0 < \kappa_{\alpha_0}$, $\beta_1 < \kappa_{\alpha_1}$, and $R \in \mathcal{R}$ such that $(\alpha_0, \beta_0) <_R (\alpha_1, \beta_1)$.}
	\end{enumerate}

	If $S = \langle \{\{\alpha\} \times \kappa_\alpha \mid \alpha \in I \}, \mathcal{R} \rangle$ is a $\lambda$-system, then we define $\mathrm{width}(S) = \max(\sup(\{\kappa_\alpha \mid \alpha \in I\}), |\mathcal{R}|)$ and $\mathrm{height}(S) = \lambda$. $S$ is a \emph{narrow $\lambda$-system} if $\mathrm{width}(S)^+ < \lambda$.

	$S$ is a \emph{strong $\lambda$-system} if it satisfies the following strengthening of (4):
	\begin{enumerate}
		\item[($4'$)]{for all $\alpha_0 < \alpha_1$, both in $I$, and for every $\beta_1 < \kappa_{\alpha_1}$, there are $\beta_0 < \kappa_{\alpha_0}$ and $R \in \mathcal{R}$ such that $(\alpha_0, \beta_0) <_R (\alpha_1,\beta_1)$.}
	\end{enumerate}

	If $R \in \mathcal{R}$, a \emph{branch of $S$ through $R$} is a set $b \subset S$ such that for all $a_0, a_1 \in b$, $a_0$ and $a_1$ are $R$-comparable. $b$ is a \emph{cofinal branch} if, for unboundedly many $\alpha \in I$, $b \cap S_\alpha \not= \emptyset$.
\end{definition}

\begin{remark}
	In previous presentations of systems (e.g. \cite{magidorshelah} and \cite{sinapova}), $\lambda$-systems were typically considered only for successor cardinals $\lambda$, and it was assumed that all $\lambda$-systems were of the form $\langle I \times \kappa, \mathcal{R} \rangle$, i.e. that all levels of the system were of the same width. If $\lambda$ is a successor cardinal and $S = \langle \{\{\alpha\} \times \kappa_\alpha \mid \alpha \in I \}, \mathcal{R} \rangle$ is a $\lambda$-system, or if $\kappa$ is weakly inaccessible and $S$ is a narrow $\lambda$-system, then there is an unbounded $J \subseteq I$ and a $\kappa < \lambda$ such that, for all $\alpha \in J$, $\kappa_\alpha = \kappa$. It will then be sufficient for us to work with subsystems of the form $\langle J \times \kappa, \mathcal{R} \rangle$, so, in the case that $\lambda$ is a successor cardinal, we will assume our systems are of this form (and typically, we will in fact have $\lambda = \kappa^+$). If $\lambda$ is weakly inaccessible and we do not want to assume narrowness, though, our more general notion of system seems to us to be the correct notion to work with.  
\end{remark}

\begin{proposition}
	Suppose $\mu$ is an infinite cardinal. Then there is a strong $\mu^+$-system $S = \langle I \times \kappa, \mathcal{R} \rangle$ such that $|\mathcal{R}| = \mu$ and $S$ has no cofinal branch.
\end{proposition}

We provide two simple and quite different proofs of this proposition.

\begin{proof}[Proof 1]
	We define a system $S = \langle \mu^+ \times 2, \mathcal{R} \rangle$, where $\mathcal{R} = \{R^i_\eta \mid i < 2, \eta < \mu \}$. For each $\alpha < \mu^+$, let $f_\alpha : \alpha \rightarrow \mu$ be injective. Fix $i < 2$ and $\eta < \mu$, and suppose $\alpha < \beta < \mu^+$ and $k_\alpha, k_\beta < 2$. Then $(\alpha, k_\alpha) <_{R^i_\eta} (\beta, k_\beta)$ if and only if:
	\begin{itemize}
		\item{$k_\beta = i$;}
		\item{$k_\alpha = 1-i$;}
		\item{$f_\beta(\alpha) = \eta$.}
	\end{itemize}
	The statement that each $R^i_\eta$ is transitive and tree-like is vacuously true, and it is easily verified that this defines a strong system. $S$ does not even have a branch of length 3, so it certainly does not have a cofinal branch.
\end{proof}

\begin{proof}[Proof 2]
	If $\mu = \omega$, then there is a $\mu^+$-Aronszajn tree, which is a strong $\mu^+$ system with 1 relation and no cofinal branch. If $\mu > \omega$, let $\bb{P} = \mathrm{Coll}(\omega, \mu)$. In $V^{\bb{P}}$, $\mu^+ = \omega_1$, so there is a $\mu^+$-Aronszajn tree. Let $\dot{T}$ be a $\bb{P}$-name for a $\mu^+$-Aronszajn tree. Without loss of generality, the underlying set of $\dot{T}$ is forced to be $\mu^+ \times \omega$. 

	In $V$, we define a system $S = \langle \mu^+ \times \omega, \mathcal{R} \rangle$, where $\mathcal{R} = \{R_p \mid p \in \bb{P}\}$. If $p \in \bb{P}$, $\alpha < \beta < \mu^+$, and $n_\alpha, n_\beta < \omega$, then let $(\alpha,n_\alpha) <_{R_p} (\beta, n_\beta)$ if and only if $p \Vdash ``(\alpha,n_\alpha) <_{\dot{T}} (\beta,n_\beta)."$ It is easily verified that $S$ is a strong $\mu^+$-system. If $S$ had a cofinal branch, there would be an unbounded set $I \subseteq \mu^+$ and a condition $p \in \bb{P}$ such that, for every $\alpha \in I$, there is $n_\alpha < \omega$ such that, whenever $\alpha < \beta$ are both in $I$, $p \Vdash ``(\alpha,n_\alpha) <_{\dot{T}} (\beta, n_\beta)."$ But then $p$ forces that the downward closure of $\{(\alpha, n_\alpha) \mid \alpha \in I\}$ is a cofinal branch in $\dot{T}$, contradicting the fact that $\dot{T}$ is a name for an Aronszajn tree.
\end{proof}

\begin{definition}
	Let $\lambda$ be a regular cardinal. $\lambda$ satisfies the \emph{strong system property} if, whenever $S = \langle I \times \kappa, \mathcal{R} \rangle$ is a strong system and $|\mathcal{R}|^+ < \lambda$, $S$ has a cofinal branch.
\end{definition}

\begin{remark}
	Note that, if $\lambda$ is a regular cardinal, then a $\lambda$-tree $(T, <_T)$ can be viewed as a strong $\lambda$-system with 1 relation. Thus, if $\lambda$ satisfies the strong system property, then $\lambda$ also satisfies the tree property.
\end{remark}

In Section \ref{reflSect}, we saw that robust stationary reflection is equivalent to the property that every stationary set reflects at arbitrarily high cofinalities. It is not clear that we have an exactly analogous situation here with the robust tree property and the strong system property. We do, however, have the following.

\begin{proposition} \label{robustsystemproperty}
	Suppose $\lambda$ is a regular, uncountable cardinal. The following are equivalent.
	\begin{enumerate}
		\item{Every strong $\lambda$-system with only countably many relations has a cofinal branch, and this property is robust under small forcing.}
		\item{$\lambda$ satisfies the strong system property.}
	\end{enumerate}
\end{proposition}

\begin{proof}
	(1) $\Rightarrow$ (2): Suppose (1) holds, and let $S = \langle \{\{\alpha\} \times \kappa_\alpha \mid \alpha \in I \}, \mathcal{R} \rangle$ be a strong $\lambda$-system with $|\mathcal{R}|^+ < \lambda$. Let $\bb{P} = \mathrm{Coll}(\omega, |\mathcal{R}|)$, and let $G$ be $\bb{P}$-generic. $|\bb{P}|^+ < \lambda$ and, in $V[G]$, $S$ is a strong $\lambda$-system with countably many relations. Thus, by (1), there is a cofinal branch $b \subseteq S$ in $V[G]$. Let $\dot{b} \in V$ be a $\bb{P}$-name for a cofinal branch through $S$. For $p \in \bb{P}$, let $b_p = \{u \in S \mid p \Vdash ``u \in \dot{b}"\}$. Each $b_p$ is a branch through $S$. Since $|\bb{P}| < \lambda$, there is $p \in G$ such that $b_p$ is cofinal.

	(2) $\Rightarrow$ (1): Suppose $\lambda$ satisfies the strong system property and, for sake of contradiction, suppose there is a poset $\bb{P}$ such that $|\bb{P}|^+ < \lambda$, a condition $p \in \bb{P}$, and a $\bb{P}$-name $\dot{S} = \langle \{\{\alpha\} \times \dot{\kappa}_\alpha \mid \alpha \in \dot{I} \}, \{\dot{R}_n \mid n < \omega \} \rangle$ such that $p$ forces $\dot{S}$ to be a strong $\lambda$-system with countably many relations and no cofinal branch.

	For all $\alpha < \lambda$ such that $p \not\Vdash ``\alpha \not\in I,"$ find $q_\alpha \leq p$ and $\kappa^*_\alpha < \lambda$ such that $q_\alpha \Vdash ``\alpha \in \dot{I}$ and $\dot{\kappa}_\alpha = \kappa^*_\alpha."$ As $|\bb{P}| < \lambda$, we can find an unbounded $J \subseteq \lambda$ and a $q \leq p$ such that, for all $\alpha \in J$, $q_\alpha = q$. Define a system $T = \langle \{\{\alpha\} \times \kappa^*_\alpha \mid \alpha \in J\}, \{R_{n,s} \mid n < \omega, s \leq q\} \rangle$ in $V$ as follows: for all $\alpha_0 < \alpha_1$, both in $J$, for all $\beta_0 < \kappa^*_{\alpha_0}$ and $\beta_1 < \kappa^*_{\alpha_1}$, for all $n < \omega$, and for all $s \leq q$, let $(\alpha_0, \beta_0) <_{R_{n,s}} (\alpha_1, \beta_1)$ iff $s \Vdash ``(\alpha_0, \beta_0) <_{\dot{R}_n} (\alpha_1, \beta_1)"$. Since $p$ forces $\dot{S}$ to be a strong $\lambda$-system, it is easily verified that $T$ is a strong $\lambda$-system with $|\bb{P}|$ relations. By the strong system property, there are $b \subseteq T$, $n < \omega$, and $s \leq q$ such that $b$ is a cofinal branch in $T$ through $R_{n, s}$. But then $s \Vdash ``b$ is a cofinal branch in $\dot{S}$ through $\dot{R}_n$," contradicting the assumption that $p$ forces $\dot{S}$ to have no cofinal branches.
\end{proof}

\begin{remark}
	Note that, in the proof of (2) $\Rightarrow$ (1) in Proposition \ref{robustsystemproperty}, there was nothing special about the assumption that $\dot{S}$ was forced to only have countably many relations. We could have reached the same conclusion from the assumption that $p \Vdash ``\dot{S}$ has $\delta$ relations" for any $\delta$ with $\delta^+ < \lambda$. We therefore obtain that the strong system property is always robust under small forcing. In particular, suppose $\lambda$ is regular, $\kappa$ and $\mu$ are such that $\kappa^+, \mu^+ < \lambda$, $\bb{P}$ is a forcing poset with $|\bb{P}| = \kappa$, and, in $V^{\bb{P}}$, there is a strong $\lambda$-system with $\mu$ relations and no cofinal branch. Then, in $V$, there is a strong $\lambda$-system with $\max(\mu, \kappa)$ relations and no cofinal branch.
\end{remark}

\begin{proposition} \label{wkcompact}
	Suppose $\lambda$ is weakly compact. Then $\lambda$ satisfies the strong system property.
\end{proposition}

\begin{proof}
	Let $S = \langle \{\{\alpha\} \times \kappa_\alpha \mid \alpha \in I\}, \mathcal{R} \rangle$ be a strong $\lambda$-system. $S$ can be coded in a natural way by a set $A \subseteq V_\lambda$. By the weak compactness of $\lambda$, find a set $X \not= V_\lambda$ and $B \subseteq X$ such that $(V_\lambda, \in, A) \prec (X, \in, B)$. By elementarity and the fact that $|\mathcal{R}| < \lambda$, $B$ codes a strong system $T = \langle \{\{\alpha\} \times \kappa_\alpha \mid \alpha \in J\}, \mathcal{R} \rangle$ such that $J$ is unbounded in the ordinals of $X$ and $T$ extends $S$. Choose $\gamma \in J \setminus \kappa$ and, for all $\alpha \in I$, find $\beta_\alpha < \kappa_\alpha$ and $R_\alpha \in \mathcal{R}$ such that $(\alpha, \beta_\alpha) <_{R_\alpha} (\gamma, 0)$ in $T$. Since $|\mathcal{R}| < \lambda$, there is an unbounded $I^* \subseteq I$ and a fixed $R \in \mathcal{R}$ such that, for all $\alpha \in I^*$, $R_\alpha = R$. Then $b = \{(\alpha, \beta_\alpha) \mid \alpha \in I^*\}$ is a cofinal branch in $S$ through $R$.
\end{proof}

Since $\lambda$ is weakly compact iff $\lambda$ is inaccessible and has the tree property, Proposition \ref{wkcompact} implies that, for inaccessible $\lambda$, the tree property is equivalent to the strong system property. As we will see later, this equivalence does not necessarily hold for accessible cardinals. Also, note that this is in contrast to the situation with stationary reflection, as we saw in the previous section that, for inaccessible $\lambda$, stationary reflection is not necessarily equivalent to robust stationary reflection.

\begin{proposition}
	Suppose $\mu$ is a singular limit of strongly compact cardinals and $\lambda > \mu$ is a regular cardinal. If $S = \langle I \times \kappa, \mathcal{R} \rangle$ is a strong $\lambda$-system, $\kappa \leq \mu$, and $|\mathcal{R}| < \mu$, then $S$ has a cofinal branch. In particular, $\mu^+$ satisfies the strong system property.
\end{proposition}

\begin{proof}
	Fix a regular $\lambda > \mu$, and let $S = \langle I \times \kappa, \mathcal{R} \rangle$ be a strong $\lambda$-system with $\kappa \leq \mu$ and $|\mathcal{R}| < \mu$. We assume for this proof that $\kappa = \mu$, as the case $\kappa < \mu$ is easier. Let $\langle \mu_i \mid i < \cf(\mu) \rangle$ be an increasing sequence of strongly compact cardinals, cofinal in $\mu$, such that $\cf(\mu), |\mathcal{R}| < \mu_0$.

	Let $F$ be the filter of co-bounded subsets of $S$, i.e. the set of $X \subseteq I \times \kappa$ such that $|(I \times \kappa) \setminus X| < \lambda$, and let $U$ be a $\mu_0$-complete ultrafilter on $S$ extending $F$. For each $\alpha \in I$ and $u \in S_{>\alpha}$, pick $(i^\alpha_u, R^\alpha_u)$, with $i^\alpha_u < \cf(\mu)$ and $R^\alpha_u \in \mathcal{R}$, such that, for some $\beta < \mu_{i^\alpha_u}$, $(\alpha, \beta) <_{R^\alpha_u} u$. Since $S_{>\alpha} \in U$ and $U$ is $\mu_0$-complete, there is $(i^\alpha, R^\alpha)$ such that the set $X_\alpha := \{u \in S_{>\alpha} \mid (i^\alpha_u, R^\alpha_u) = (i^\alpha, R^\alpha)\} \in U$. There is then an unbounded $J \subseteq I$ and $(i^*, R^*)$ such that, for all $\alpha \in J$, $(i^\alpha, R^\alpha) = (i^*, R^*)$. Now, if $\alpha_0 < \alpha_1$ are both in $J$, we can find $u \in X_{\alpha_0} \cap X_{\alpha_1}$. There are then $\beta_0, \beta_1 < \mu_{i^*}$ such that $(\alpha_0, \beta_0), (\alpha_1, \beta_1) <_{R^*} u$. Since $R^*$ is tree-like, we have $(\alpha_0, \beta_0) <_{R^*} (\alpha_1, \beta_1)$. This shows that $S' := \langle J \times \mu_{i^*}, \{R^*\} \rangle$ is a $\lambda$-system.

	Next, fix $k > i^*$ with $k < \cf(\mu)$, and let $U'$ be a $\mu_k$-complete ultrafilter over $\lambda$ extending the co-bounded filter and such that $J \in U'$. Fix $\alpha \in J$. For all $\beta \in J \setminus (\alpha + 1)$, fix $\gamma^\alpha_\beta, \delta^\alpha_\beta < \mu_{i^*}$ such that $(\alpha, \gamma^\alpha_\beta) <_{R^*} (\beta, \delta^\alpha_\beta)$. Since $J \setminus (\alpha + 1) \in U'$ and $U'$ is $\mu_k$-complete, we can fix $\gamma^\alpha, \delta^\alpha < \mu_{i^*}$ such that $Y_\alpha := \{ \beta \in J \setminus (\alpha + 1) \mid (\gamma^\alpha_\beta, \delta^\alpha_\beta) = (\gamma^\alpha, \delta^\alpha)\} \in U'$. Next, fix an unbounded $J' \subseteq J$ and $\gamma^*, \delta^* < \mu_{i^*}$ such that, for all $\alpha \in J'$, $(\gamma^\alpha, \delta^\alpha) = (\gamma^*, \delta^*)$. Suppose $\alpha_0 < \alpha_1$ are both in $J'$. Fix $\beta \in X_{\alpha_0} \cap X_{\alpha_1}$. Then $(\alpha_0, \gamma^*), (\alpha_1, \gamma^*) <_{R^*} (\beta, \delta^*)$, so, since $R^*$ is tree-like, $(\alpha_0, \gamma^*) <_{R^*} (\alpha_1, \gamma^*)$. Hence, $\langle (\alpha, \gamma^*) \mid \alpha \in J' \rangle$ is a cofinal branch through $R^*$ in $S$. 
\end{proof}

A similar argument shows that all systems with finite width have a cofinal branch.

\begin{proposition}
	Suppose $\lambda$ is a regular cardinal and $S = \langle I \times n, \mathcal{R} \rangle$ is a $\lambda$-system with $n, |\mathcal{R}| < \omega$. Then $S$ has a cofinal branch.
\end{proposition}

\begin{proof}
	Let $U$ be an ultrafilter over $\lambda$, extending the co-bounded filter, such that $I \in U$. Fix $\alpha \in I$. For all $\beta \in I \setminus (\alpha + 1)$, choose $i^\alpha_\beta, j^\alpha_\beta < n$ and $R^\alpha_\beta \in \mathcal{R}$ such that $(\alpha, i^\alpha_\beta) <_{R^\alpha_\beta} (\beta, j^\alpha_\beta)$. Fix $i^\alpha, j^\alpha < n$ and $R^\alpha \in \mathcal{R}$ such that $X_\alpha := \{\beta \in I \setminus (\alpha + 1) \mid (i^\alpha_\beta, j^\alpha_\beta, R^\alpha_\beta) = (i^\alpha, j^\alpha, R^\alpha)\} \in U$. Fix an unbounded $J \subseteq I$ and $(i^*, j^*, R^*)$ such that, for all $\alpha \in J$, $(i^\alpha, j^\alpha, R^\alpha) = (i^*, j^*, R^*)$. Now, suppose $\alpha_0 < \alpha_1$ are both in $J$, and find $\beta \in X_{\alpha_0} \cap X_{\alpha_1}$. $(\alpha_0, i^*), (\alpha_1, i^*) <_{R^*} (\beta, j^*)$, so $(\alpha_0, i^*) <_{R^*} (\alpha_1, j^*)$. Thus, $\{(\alpha, i^*) \mid \alpha \in J\}$ is a cofinal branch through $R^*$ in $S$.
\end{proof}

However, the existence of certain subadditive, unbounded functions implies the existence of strong systems of possibly small width with no cofinal branch.

\begin{definition}
	Suppose $\kappa < \lambda$ are infinite, regular cardinals and $d:[\lambda]^2 \rightarrow \kappa$.
	\begin{enumerate}
		\item{$d$ is \emph{subadditive} if, for all $\alpha < \beta < \gamma < \lambda$:}
		\begin{enumerate}
			\item{$d(\alpha, \gamma) \leq \max(\{d(\alpha, \beta), d(\beta, \gamma)\})$;}
			\item{$d(\alpha, \beta) \leq \max(\{d(\alpha, \gamma), d(\beta, \gamma)\})$.}
		\end{enumerate}
		\item{$d$ is \emph{unbounded} if, whenever $I \subseteq \lambda$ is unbounded, $d``[I]^2$ is unbounded in $\kappa$.}
	\end{enumerate}
\end{definition}

For more on the consistency of such functions, see \cite{narrowsystems}.

\begin{proposition}
	Suppose $\kappa < \lambda$ are infinite, regular cardinals and $d:[\lambda]^2 \rightarrow \kappa$ is subadditive and unbounded. Then there is a strong $\lambda$-system $S = \langle \lambda \times 1, \mathcal{R} \rangle$ such that $|\mathcal{R}| = \kappa$ and $S$ has no cofinal branch.
\end{proposition}

\begin{proof}
	We define $S = \langle \lambda \times 1, \mathcal{R} \rangle$, with $\mathcal{R} = \{R_\eta \mid \eta < \kappa\}$. Given $\alpha < \beta < \lambda$ and $\eta < \kappa$, let $(\alpha,0) <_{R_\eta} (\beta,0)$ if and only if $\eta \geq d(\alpha, \beta)$. The fact that each $R_\eta$ is transitive and tree-like follows from (a) and (b) of the definition of subadditivity, respectively. It is then easy to verify that $S$ is a strong $\lambda$-system. Suppose for sake of contradiction that $S$ has a cofinal branch. Then there is an unbounded $I \subseteq \lambda$ and an $\eta < \kappa$ such that, for all $\alpha < \beta$, both in $I$, we have $(\alpha,0) <_{R_\eta} (\beta,0)$. Since $d$ is unbounded, we can find $\alpha < \beta$ in $I$ such that $d(\alpha, \beta) > \eta$. But then $(\alpha,0) \not<_{R_\eta} (\beta,0)$. Contradiction.
\end{proof}

\section{Preservation lemmas} \label{preservationSect}

In this section, we present various preservation lemmas for systems. We first make the following useful definition.

\begin{definition}
	Let $\lambda$ be an uncountable regular cardinal, and let $S = \langle I \times \kappa, \mathcal{R} \rangle$ be a narrow $\lambda$-system. $\bar{b} = \{b_{\gamma, R} \mid \gamma < \kappa, R \in \mathcal{R} \}$ is a \emph{full set of branches through S} if:
	\begin{enumerate}
		\item{for all $\gamma < \kappa$ and $R \in \mathcal{R}$, $b_{\gamma, R}$ is a branch of $S$ through $R$;}
		\item{for all $\alpha \in I$, there are $\gamma < \kappa$ and $R \in \mathcal{R}$ such that $b_{\gamma, R} \cap S_\alpha \not= \emptyset$.}
	\end{enumerate}
\end{definition}

\begin{remark}
	Note that, since $\lambda$ is regular and $\mathrm{width}(S) < \lambda$, condition (2) in the above definition implies that, for some $\gamma < \kappa$ and $R \in \mathcal{R}$, $b_{\gamma, R}$ is a cofinal branch.
\end{remark}

\begin{lemma} \label{preservationtheorem}
	Suppose $\theta < \lambda$ are regular cardinals and $\bb{P}$ and $\bb{Q}$ are forcing notions such that $\bb{P}$ has the $\theta$-c.c.\ and $\bb{Q}$ is $\theta$-closed. Let $G$ be $\bb{P}$-generic over $V$ and let $H$ be $\bb{Q}$-generic over $V[G]$. Suppose further that, in $V[G]$, $S$ is a narrow $\lambda$-system, $\mathrm{width}(S) < \theta$, and, in $V[G][H]$, $S$ has a full set of branches. Then $S$ has a cofinal branch in $V[G]$.
\end{lemma}

\begin{remark}
	This is a slight improvement of Theorem 8 from \cite{sinapova}. The difference is that our proof works for any $S$ with $\mathrm{width}(S) < \theta$, whereas the original theorem required $\mathrm{width}(S)^+ < \theta$.
\end{remark}

\begin{proof}
	Work in $V$. Let $\dot{S}$ be a $\bb{P}$-name for a narrow $\lambda$-system as in the statement of the lemma. Since every narrow $\lambda$-system is easily isomorphic to one of the form $\langle \lambda \times \kappa, \mathcal{R} \rangle$, we may choose $(p,q) \in \bb{P} \times \bb{Q}$ such that:
	\begin{itemize}
		\item{there are $\kappa, \nu < \theta$ such that $p$ forces $\dot{S}$ to be of the form $\langle \lambda \times \kappa, \{\dot{R}_\eta \mid \eta < \nu\} \rangle$;}
		\item{there is a set of $\bb{P} \times \bb{Q}$-names $\{\dot{b}_{\gamma, \dot{R}_\eta} \mid \gamma < \kappa, \eta < \nu\}$ that is forced by $(p,q)$ to be a full set of branches through $\dot{S}$ (for notational simplicity, we will write $\dot{b}_{\gamma, \eta}$ instead of $\dot{b}_{\gamma, \dot{R}_\eta}$);}
		\item{there is $\alpha^* < \lambda$ such that, for every $(\gamma, \eta) \in \kappa \times \nu$, $p \Vdash ``q \Vdash ``\dot{b}_{\gamma, \eta} \subseteq \dot{S}_{<\alpha^*}"$ or $q \Vdash ``\dot{b}_{\gamma, \eta}$ is cofinal.""}
	\end{itemize}
	The last requirement is possible because, in $V^{\bb{P}}$, $\bb{Q}$ is $\theta$-distributive and $\max(\{\kappa, \nu\}) < \theta$. Also, for sake of contradiction, assume that $p \Vdash ``\dot{S}$ has no cofinal branches." Let $\epsilon := \max(\{\kappa, \nu\})$.

	\begin{claim} \label{splitclaim1}
		Suppose  $(\gamma, \eta) \in \kappa \times \nu$, $p' \leq p$, $(p',q) \Vdash ``\dot{b}_{\gamma, \eta}$ is cofinal," and $q_0, q_1 \leq q.$ Then there are $p'' \leq p'$, $q''_i \leq q_i$, and $u_i \in \lambda \times \kappa$ for $i < 2$ such that:
		\begin{itemize}
			\item{for $i < 2$, $(p'', q''_i) \Vdash ``u_i \in \dot{b}_{\gamma, \eta}"$;}
			\item{$p'' \Vdash ``u_0 \perp_{\dot{R}_\eta} u_1."$}
		\end{itemize}
	\end{claim}

	\begin{proof}
		Let $G'$ be $\bb{P}$-generic over $V$ with $p' \in G'$, and work in $V[G']$. For $i < 2$, let $E_i = \{u \in \lambda \times \kappa \mid$ for some $q^* \leq q_i$, $q^* \Vdash ``u \in \dot{b}_{\gamma, \eta}"\}.$ Since, in $V$, $(p', q) \Vdash ``\dot{b}_{\gamma, \eta}$ is cofinal," $\{\alpha < \lambda \mid E_i \cap S_\alpha \not= \emptyset\}$ is cofinal in $\lambda$. Since $S$ has no cofinal branch in $V[G']$, we can find $u^0_0, u^1_0 \in E_0$ such that $u^0_0 \perp_{R_\eta} u^1_0$. Let $\alpha < \lambda$ be such that $u^0_0, u^1_0 \in S_{<\alpha}$, and find $u_1 \in E_1 \cap S_{\geq \alpha}$. Since $R_\eta$ is tree-like, $u_1$ is $R_\eta$-comparable with at most one of $u^0_0$ and $u^1_0$. Let $k < 2$ be such that $u^k_0 \perp_{R_\eta} u_1$, and let $u_0 = u^k_0$. For $i < 2$, find $q''_i \leq q_i$ such that $q''_i \Vdash ``u_i \in \dot{b}_{\gamma, \eta}."$ Finally, find $p'' \in G'$ such that $p'' \leq p'$, $p'' \Vdash ``u_0 \perp_{\dot{R}_\eta} u_1,"$ and, for $i < 2$, $(p'', q''_i) \Vdash ``u_i \in \dot{b}_{\gamma, \eta}."$ Then $p''$, $q''_i$, and $u_i$ are as desired.
	\end{proof}

	\begin{claim} \label{splitclaim2}
		Suppose $(\gamma, \eta) \in \kappa \times \nu$ and $q_0, q_1 \leq q$. Then there are $q'_i \leq q_i$ for $i < 2$ and a maximal antichain $A \subseteq \bb{P}$ of conditions below $p$ such that, for all $p' \in A$, either $(p',q) \Vdash ``\dot{b}_{\gamma, \eta}$ is not cofinal" or there are $u_i \in \lambda \times \kappa$ for $i < 2$ such that:
		\begin{itemize}
			\item{for $i < 2$, $(p', q'_i) \Vdash ``u_i \in \dot{b}_{\gamma, \eta}"$;}
			\item{$p' \Vdash ``u_0 \perp_{\dot{R}_\eta} u_1."$}
		\end{itemize}
	\end{claim}

	\begin{proof}
		If $(p,q) \Vdash ``\dot{b}_{\gamma, \eta}$ is not cofinal," we may simply take $A = \{p\}$ and, for $i < 2$, $q'_i = q_i$. Otherwise, we recursively attempt to construct $\{p_\xi \mid \xi < \theta\}$ and $\{q^\xi_i \mid \xi < \theta, i < 2\}$, with $\{p_\xi \mid \xi < \theta\}$ an antichain in $\bb{P}$ below $p$ and, for $i < 2$, $\langle q^\xi_i \mid \xi < \theta \rangle$ a decreasing sequence of conditions in $\bb{Q}$ below $q_i$. To start, let $p'_0 \leq p$ be such that $(p'_0, q) \Vdash ``\dot{b}_{\gamma, \eta}$ is cofinal." Then apply Claim \ref{splitclaim1} to $p'_0$ and $q_i$ for $i < 2$ to obtain $p_0 \leq p'_0$, $q^0_i \leq q_i$, and $u^0_i \in \lambda \times \kappa$ for $i < 2$. 

		Suppose $\xi < \theta$ and we have defined $\{p_\zeta \mid \zeta \leq \xi\}$ and $\{q^\zeta_i \mid \zeta \leq \xi, i < 2\}$. If $\{p_\zeta \mid \zeta \leq \xi\}$ is a maximal antichain below $p$, then stop the construction, let $A = \{p_\zeta \mid \zeta \leq \xi\}$, and, for $i < 2$, let $q'_i = q^\xi_i$. Otherwise, find $p'_{\xi+1} \leq p$ such that $p'_{\xi+1}$ is incompatible with $p_\zeta$ for every $\zeta \leq \xi$. If $(p'_{\xi+1},q) \Vdash ``\dot{b}_{\gamma, \eta}$ is not cofinal," then let $p_{\xi+1} = p'_{\xi+1}$ and, for $i < 2$, let $q^{\xi+1}_i = q^\xi_i$. Otherwise, find $p''_{\xi+1} \leq p'_{\xi+1}$ such that $(p''_{\xi+1}, q) \Vdash ``\dot{b}_{\gamma, \eta}$ is cofinal," and apply Claim \ref{splitclaim1} to $p''_{\xi+1}$ and $q^\xi_i$ for $i < 2$ to obtain $p_{\xi+1} \leq p''_{\xi+1}$, $q^{\xi+1}_i \leq q^\xi_i$, and $u^{\xi+1}_i \in \lambda \times \kappa$ for $i < 2$.

		Finally, suppose $\xi < \theta$ is a limit ordinal and we have defined $\{p_\zeta \mid \zeta < \xi\}$ and $\{q^\zeta_i \mid \zeta < \xi, i < 2\}$. For $i < 2$, let $q^*_i$ be a lower bound for $\{q^\zeta_i \mid \zeta < \xi\}$. If $\{p_\zeta \mid \zeta < \xi\}$ is a maximal antichain below $p$, then stop the construction, let $A = \{p_\zeta \mid \zeta < \xi\}$, and, for $i < 2$, let $q'_i = q^*_i$. Otherwise, proceed as in the successor case (working with $q^*_i$ instead of the $q^\xi_i$ from the previous paragraph) to define $p_\xi$ and $q^\xi_i$ for $i < 2$.

		Since $\bb{P}$ has the $\theta$-c.c., our construction must halt at some stage $\xi < \theta$. It is easy to verify that $A$ and $q'_i$ for $i < 2$ are as desired.
	\end{proof}

	\begin{claim} \label{splitclaim3}
		Suppose $q' \leq q$. Then there are $q^*_0, q^*_1 \leq q'$ such that, for every $(\gamma, \eta) \in \kappa \times \nu$, there is a maximal antichain $A_{\gamma, \eta}$ of conditions below $p$ such that, for all $p' \in A_{\gamma, \eta}$, either $(p',q) \Vdash ``\dot{b}_{\gamma, \eta}$ is not cofinal" or there are $u_i \in \lambda \times \kappa$ for $i < 2$ such that:
		\begin{itemize}
			\item{for $i < 2$, $(p', q^*_i) \Vdash ``u_i \in \dot{b}_{\gamma, \eta}"$;}
			\item{$p' \Vdash ``u_0 \perp_{\dot{R}_\eta} u_1."$}
		\end{itemize}
	\end{claim}

	\begin{proof}
		Let $\{(\gamma_\xi, \eta_\xi) \mid \xi < \epsilon\}$ be an enumeration of $\kappa \times \nu$. For $i < 2$, define decreasing sequences $\langle q^\xi_i \mid \xi < \epsilon \rangle$ of conditions in $\bb{Q}$ as follows. Let $q^0_0 = q^0_1 = q'$. If $\xi < \epsilon$ is a limit ordinal and $\langle q^\zeta_i \mid \zeta < \xi, i < 2 \rangle$ has been defined, let $q^\xi_i$ be a lower bound for $\langle q^\zeta_i \mid \zeta < \xi \rangle$. If $\xi < \epsilon$ and $q^\xi_i$ has been defined for $i < 2$, then apply Claim \ref{splitclaim2} to $(\gamma_\xi, \eta_\xi)$ and $q^\xi_i \leq q$ to obtain $q^{\xi+1}_i \leq q^\xi_i$ and a maximal antichain $A_{\gamma_\xi, \eta_\xi}$ of conditions below $p$. Finally, for $i < 2$, let $q^*_i$ be a lower bound for $\langle q^\xi_i \mid \xi < \epsilon \rangle$. It is clear that $q^*_i$ for $i < 2$ are as desired, as witnessed by the antichains $\{A_{\gamma_\xi, \eta_\xi} \mid \xi < \epsilon\}$.
	\end{proof}

	We now construct, in $V$, a decreasing sequence $\langle q_\xi \mid \xi < \theta \rangle$ of conditions from $\bb{Q}$, together with conditions $\{r_{\xi + 1} \mid \xi < \theta \}$, also in $\bb{Q}$. To start, let $q_0 = q$. If $\xi < \theta$ is a limit ordinal and $\langle q_\zeta \mid \zeta < \xi \rangle$ has been defined, let $q_\xi$ be a lower bound for $\langle q_\zeta \mid \zeta < \xi \rangle$. If $\xi < \theta$ and $q_\xi$ has been defined, apply Claim \ref{splitclaim3} to $q_\xi$ to obtain $q_{\xi+1}, r_{\xi+1} \leq q_\xi$ and, for every $(\gamma, \eta) \in \kappa \times \nu$, a maximal antichain $A^\xi_{\gamma, \eta}$ of conditions from $\bb{P}$ below $p$. Find $\beta^* < \lambda$ large enough so that $\beta^* > \alpha^*$ and, for all $\xi < \theta$, all $(\gamma, \eta) \in \kappa \times \nu$, and all $p' \in A^\xi_{\gamma, \eta}$, either $(p',q) \Vdash ``\dot{b}_{\gamma, \eta}$ is not cofinal" or there are $u_0, u_1 \in S_{<\beta^*}$ such that:
	\begin{itemize}
		\item{$(p', q_{\xi+1}) \Vdash ``u_0 \in \dot{b}_{\gamma, \eta}"$;}
		\item{$(p', r_{\xi+1}) \Vdash ``u_1 \in \dot{b}_{\gamma, \eta}"$;}
		\item{$p' \Vdash ``u_0 \perp_{\dot{R}_\eta} u_1."$}
	\end{itemize}

	Let $G$ be $\bb{P}$-generic over $V$ with $p \in G$, and move to $V[G]$. For $\eta < \nu$, let $R_\eta$ be the realization of $\dot{R}_\eta$. For all $\xi < \theta$, find $r'_{\xi+1} \leq r_{\xi+1}$, $\delta_{\xi+1}, \gamma_{\xi+1} < \kappa$, and $\eta_{\xi+1} < \nu$ such that $r'_{\xi+1} \Vdash ``(\beta^*, \delta_{\xi+1}) \in \dot{b}_{\gamma_{\xi+1}, \eta_{\xi+1}}."$ Since $\theta > \epsilon$ and $\theta$ remains regular in $V[G]$, we can find $\zeta < \xi < \theta$, $\delta^*, \gamma^* < \kappa$, and $\eta^* < \nu$ such that $\delta_{\zeta+1} = \delta_{\xi+1} = \delta^*$, $\gamma_{\zeta+1} = \gamma_{\xi+1} = \gamma^*$, and $\eta_{\zeta+1} = \eta_{\xi+1} = \eta^*$. Find $p' \in A^\zeta_{\gamma^*, \eta^*} \cap G$. Since, in $V[G]$, $r'_{\zeta + 1} \Vdash ``(\beta^*, \delta^*) \in \dot{b}_{\gamma^*, \eta^*} \cap S_{\geq \alpha^*}"$, it cannot be the case that, in $V$, $(p',q) \Vdash ``\dot{b}_{\gamma^*, \eta^*}$ is not cofinal." Therefore, there are $u_0, u_1 \in S_{<\beta^*}$ such that:
	\begin{itemize}
		\item{$(p', q_{\zeta+1}) \Vdash ``u_0 \in \dot{b}_{\gamma^*, \eta^*}"$;}
		\item{$(p', r_{\zeta+1}) \Vdash ``u_1 \in \dot{b}_{\gamma^*, \eta^*}"$;}
		\item{$p' \Vdash ``u_0 \perp_{\dot{R}_{\eta^*}} u_1."$}
	\end{itemize}
	Since $p' \in G$, $r'_{\zeta + 1} \leq r_{\zeta + 1}$, and $r'_{\xi+1} \leq r_{\xi+1} \leq q_\xi \leq q_{\zeta + 1}$, we have the following in $V[G]$:
	\begin{enumerate}
		\item{$u_0 \perp_{R_{\eta^*}} u_1$;}
		\item{$r'_{\xi+1} \Vdash ``u_0, (\beta^*, \delta^*) \in \dot{b}_{\gamma^*, \eta^*}"$;}
		\item{$r'_{\zeta+1} \Vdash ``u_1, (\beta^*, \delta^*) \in \dot{b}_{\gamma^*, \eta^*}"$.}
	\end{enumerate}
	By (2), (3), and the fact that $R_{\eta^*}$ is tree-like, we have that $u_0$ and $u_1$ are $R_{\eta^*}$-comparable, but this contradicts (1), so we are finished.
\end{proof}

We also make note for future use of the following improvement of Sinapova's lemma, due to Neeman. A proof can be found in \cite{neeman}.

\begin{lemma} \label{neemanlemma}
	Suppose that $\lambda$ is a regular, uncountable cardinal, $S = \langle I \times \kappa, \mathcal{R} \rangle$ is a narrow $\lambda$-system, and $\mathrm{width}(S) = \theta$. Suppose $\bb{P}$ is a forcing poset, and let $\bb{P}^{\theta^+}$ denote the full-support product of $\theta^+$ copies of $\bb{P}$. Suppose moreover that $\bb{P}^{\theta^+}$ is $\theta^{++}$-distributive, $G$ is $\bb{P}$-generic over $V$, and, in $V[G]$, there is a full set of branches through $S$. Then there is a cofinal branch through $S$ in $V$.
\end{lemma}

We show now that cofinal branches cannot be added to systems by forcing posets satisfying the appropriate approximation property.

\begin{definition}
	Let $\lambda$ be a regular cardinal, and let $\bb{P}$ be a forcing poset. $\bb{P}$ has the \emph{$\lambda$-approximation property} if, for every $y \in V$ and every $\bb{P}$-name $\dot{x}$ for a subset of $y$ such that, for all $z \in (\mathcal{P}_\lambda(y))^V$, $\Vdash_{\bb{P}}``\dot{x} \cap z \in V"$, we have $\Vdash_{\bb{P}} \dot{x} \in V$. 
\end{definition}

\begin{lemma} \label{approxbranchprop}
	Suppose $\lambda$ is a regular cardinal, $S = \langle \{\{\alpha\} \times \kappa_\alpha \mid \alpha \in I\}, \mathcal{R} \rangle$ is a $\lambda$-system, and $\bb{P}$ has the $\lambda$-approximation property. If $G$ is $\bb{P}$-generic over $V$ and, in $V[G]$, $S$ has a cofinal branch, then $S$ has a cofinal branch in $V$.
\end{lemma}

\begin{proof}
	In $V[G]$, suppose $b \subseteq S$ is a cofinal branch through $R \in \mathcal{R}$. By closing $b$ downward, we may assume that $b$ is a maximal branch, i.e., if $\alpha \in I$ and $v \in b \cap S_\alpha$, then $b \cap S_{<\alpha} = \{u \in S \mid u <_R v\}$. It suffices to show that $b \cap z \in V$ for all $z \in (\mathcal{P}_\lambda(S))^V$, as then the $\lambda$-approximation property will imply that $b \in V$.

	To this end, let $z \in (\mathcal{P}_\lambda(S))^V$. As $\lambda$ is regular, there is $\alpha < \lambda$ such that $z \subseteq S_{<\alpha}$. Find $v \in b \cap S_{\geq \alpha}$. Then $b \cap z = \{u \in S \mid u <_R v\} \cap z \in V$.
\end{proof}

The following lemma was proved by Unger in \cite{unger}.

\begin{lemma} \label{approxlemma}
	Suppose $\lambda$ is a regular cardinal, $\bb{P}$ is a forcing poset, and $\bb{P} \times \bb{P}$ has the $\lambda$-c.c. Then $\bb{P}$ has the $\lambda$-approximation property.
\end{lemma}

\section{Collapsing and the narrow system property} \label{collapsingSect}

\begin{definition}
	Suppose $\lambda$ is a regular cardinal. $\lambda$ has the \emph{narrow system property} if every narrow $\lambda$-system has a cofinal branch.
\end{definition}

The narrow system property is a useful tool for analyzing trees and strong systems. For more on the narrow system property, see \cite{narrowsystems}. In this section, we prove a technical lemma about the narrow system property in a forcing extension by a product of Levy collapses.

\begin{lemma} \label{collapselemma}
	Suppose $\langle \kappa_n \mid n < \omega \rangle$ is an increasing sequence of indestructibly supercompact cardinals. Let $\mu = \sup(\{\kappa_n \mid n < \omega\})$, and let $\lambda = \mu^+$. For $m < \omega$, let $\bb{S}_m$ be the full-support product $\prod_{n \geq m} \mathrm{Coll}(\kappa^{+2}_n, < \kappa_{n+1})$. Then, in $V^{\bb{S}_m}$, $\lambda$ has the narrow system property.
\end{lemma}

\begin{proof}
	For $n > m$, let $\bb{S}_{m,n} = \prod_{m\leq k < n} \mathrm{Coll}(\kappa^{+2}_k, < \kappa_{k+1})$, and note that $\bb{S}_m \cong \bb{S}_{m,n} \times \bb{S}_{n}$. Let $\dot{S}$ be an $\bb{S}_m$-name for a narrow $\lambda$-system. Without loss of generality, we may assume there are $\kappa, \nu < \mu$ such that $\dot{S}$ is forced to be of the form $\langle \lambda \times \kappa, \{\dot{R}_\eta \mid \eta < \nu\} \rangle$. Fix $m < n^* < \omega$ such that $\kappa, \nu < \kappa_{n^*}$. Let $G_{n^*+1}$ be $\bb{S}_{n^*+1}$-generic over $V$. Note that, in $V[G_{n^*+1}]$, $\kappa_{n^*+1}$ remains supercompact. Move to $V[G_{n^*+1}]$, which we denote $V_1$, reinterpreting $\dot{S}$ as an $\bb{S}_{m,n^*+1}$-name.

	Let $j:V_1 \rightarrow M$ witness that $\kappa_{n^*+1}$ is $\lambda$-supercompact. Note that $j(\bb{S}_{m,n^*}) = \bb{S}_{m,n^*}$ and $j(\bb{S}_{n^*, n^*+1}) = \mathrm{Coll}(\kappa^{+2}_{n^*}, < j(\kappa_{n^*+1})) \cong \bb{S}_{n^*, n^*+1} * \dot{\bb{R}}$, where $\dot{\bb{R}}$ is $\kappa^{+2}_{n^*}$-closed in $V_1^{\bb{S}_{n^*, n^*+1}}$. Let $G_{n^*, n^*+1}$ be $\bb{S}_{n^*, n^*+1}$-generic over $V_1$, and let $\bb{R}$ be the interpretation of $\dot{\bb{R}}$ in $V_1[G_{n^*, n^*+1}]$. Then, letting $G_{m,n^*}$ be $\bb{S}_{m,n^*}$-generic over $V_1[G_{n^*, n^*+1}]$ and $H$ be $\bb{R}$-generic over $V_1[G_{n^*, n^*+1}][G_{m,n^*}]$, we can lift $j$ to $j:V_1[G_{n^*, n^*+1}][G_{m,n^*}] \rightarrow M[G_{n^*, n^*+1}][G_{m,n^*}][H]$.

	Let $S = \langle \lambda \times \kappa, \{R_\eta \mid \eta < \nu\} \rangle$ be the realization of $\dot{S}$ in $V_1[G_{n^*, n^*+1}][G_{m,n^*}]$. In $M[G_{n^*, n^*+1}][G_{m,n^*}][H]$, $j(S) = \langle j(\lambda) \times \kappa, \{j(R_\eta) \mid \eta < \nu\} \rangle$ is a $j(\lambda)$-system. Let $\delta = \sup(j``\lambda) < j(\lambda)$. For all $\gamma < \kappa$ and $\eta < \nu$, let $b_{\gamma, R_\eta} = \{(\alpha, \beta) \in \lambda \times \kappa \mid (j(\alpha), \beta) <_{j(R_\eta)} (\delta, \gamma)\}$. By elementarity and the fact that $j(S)$ is a system, it is easily seen that $\bar{b} := \{b_{\gamma, R_\eta} \mid \gamma < \kappa, \eta < \nu\}$ is a full set of branches through $S$ and $\bar{b} \in V_1[G_{n^*, n^*+1}][G_{m,n^*}][H]$. In $V_1[G_{n^*, n^*+1}]$, $\bb{S}_{m,n^*}$ has the $\kappa_{n^*}$-c.c. and $\bb{R}$ is $\kappa_{n^*}^{+2}$-closed, so, by Lemma \ref{preservationtheorem}, $S$ has a cofinal branch in $V_1[G_{n^*, n^*+1}][G_{m,n^*}]$.
\end{proof}

\begin{corollary}
	Assume the same hypotheses as in Lemma \ref{collapselemma}. In $V^{\bb{S}_m}$, suppose $S = \langle \lambda \times \mu, \mathcal{R} \rangle$ is a strong $\lambda$-system and $|\mathcal{R}| < \kappa_m$. Then $S$ has a cofinal branch.
\end{corollary}

\begin{proof}
	Let $G$ be $\bb{S}_m$-generic over $V$. In $V[G]$, $\kappa_m$ remains supercompact. Let $j:V[G] \rightarrow M$ witness that $\kappa_m$ is $\lambda$-supercompact. In $M$, $j(S) = \langle j(\lambda) \times j(\mu), \{j(R) \mid R \in \mathcal{R}\} \rangle$ is a strong $j(\lambda)$-system. Let $\delta = \sup(j``\lambda)$. For $\alpha < \lambda$, find $\beta_\alpha < j(\mu)$ and $R_\alpha \in \mathcal{R}$ such that $(j(\alpha), \beta_\alpha) <_{j(R_\alpha)} (\delta, 0)$. Let $n_\alpha < \omega$ be such that $\beta_\alpha < j(\kappa_{n_\alpha})$. Since $\lambda$ is regular, we can find $R^* \in \mathcal{R}$, $n^* < \omega$, and an unbounded $I \subseteq \lambda$ such that, for all $\alpha \in I$, $R_\alpha = R^*$ and $n_\alpha = n^*$. Now, if $\alpha_0 < \alpha_1$ are both in $I$, then $(j(\alpha_0), \beta_{\alpha_0}), (j(\alpha_1), \beta_{\alpha_1}) <_{j(R^*)} (\delta, 0)$, so $(j(\alpha_0), \beta_{\alpha_0}) <_{j(R^*)} (j(\alpha_1), \beta_{\alpha_1})$. Thus, $M \models ``$there are $\beta_0, \beta_1 < j(\kappa_{n^*})$ such that $(j(\alpha_0), \beta_0) <_{j(R^*)} (j(\alpha_1), \beta_1),"$ so, by elementarity, $V[G] \models ``$there are $\beta_0, \beta_1 < \kappa_{n^*}$ such that $(\alpha_0, \beta_0) <_{R^*} (\alpha_1, \beta_1)."$ Therefore, in $V[G]$, $S' = \langle I \times \kappa_{n^*}, \{R^*\} \rangle$ is a narrow system. By Lemma \ref{collapselemma}, $S'$ has a cofinal branch, $b$, which is then also a cofinal branch of $S$.
\end{proof}

\section{Weak square and strong systems} \label{weakSquareSect}

Recall the following definition.

\begin{definition}
	Let $\lambda$ and $\mu$ be cardinals, with $\mu$ infinite and $\lambda > 1$. A $\square_{\mu, <\lambda}$-sequence is a sequence $\vec{\mathcal{C}} = \langle \mathcal{C}_\alpha \mid \alpha < \mu^+ \rangle$ such that:
	\begin{enumerate}
		\item{for all limit $\alpha < \mu^+$, if $C \in \mathcal{C}_\alpha$, then $C$ is a club in $\alpha$ and $\mathrm{otp}(C) \leq \mu$;}
		\item{for all limit $\alpha < \mu^+$, $1 \leq |\mathcal{C}_\alpha| < \lambda$;}
		\item{for all limit $\alpha < \beta < \mu^+$ and all $C \in \mathcal{C}_\beta$, if $\alpha \in C'$, then $C \cap \alpha \in \mathcal{C}_\alpha$.}
	\end{enumerate}
	$\square_{\mu, < \lambda}$ holds if there is a $\square_{\mu, < \lambda}$-sequence.
\end{definition}

\begin{remark}
	$\square_{\mu, < \lambda^+}$ is usually denoted $\square_{\mu, \lambda}$. It is immediate that, if $\lambda_0 < \lambda_1$, then $\square_{\lambda_0}$ implies $\square_{\lambda_1}$. $\square_{\mu,1}$ is Jensen's classical principle $\square_\mu$. $\square_{\mu, \mu}$ is also called \emph{weak square} and denoted $\square^*_\mu$. $\square^*_\mu$ is equivalent to the existence of a special $\mu^+$-Aronszajn tree. $\square_{\mu, \mu^+}$ is also called \emph{silly square} and holds in all models of ZFC. 
\end{remark}

We will be interested in the following variation on the classical square principles.

\begin{definition}
	Let $\kappa$, $\lambda$, and $\mu$ be cardinals, with $\kappa \leq \mu$, $\kappa$ regular, and $\lambda > 1$. A $\square^{\geq \kappa}_{\mu, <\lambda}$-sequence is a sequence $\vec{\mathcal{C}} = \langle \mathcal{C}_\alpha \mid \alpha \in S \rangle$ such that:
	\begin{enumerate}
		\item{$S^{\mu^+}_{\geq \kappa} \subseteq S \subseteq \lim(\mu^+)$;}
		\item{for all $\alpha \in S$ and all $C \in \mathcal{C}_\alpha$, $C$ is a club in $\alpha$ and $\mathrm{otp}(C) \leq \mu$;}
		\item{for all $\alpha \in S$, $1 \leq |\mathcal{C}_\alpha| < \lambda$;}
		\item{for all $\beta \in S$, for all $C \in \mathcal{C}_\beta$, and for all $\alpha \in C'$, we have $\alpha \in S$ and $C \cap \alpha \in \mathcal{C}_\alpha$.}
	\end{enumerate}
	$\square^{\geq \kappa}_{\mu, < \lambda}$ holds if there is a $\square^{\geq \kappa}_{\mu, < \lambda}$-sequence. As usual, we shall write $\square^{\geq \kappa}_{\mu, \lambda}$ instead of $\square^{\geq \kappa}_{\mu, < \lambda^+}$.
\end{definition}

Baumgartner was the first to study such square sequences. We describe now a forcing poset designed to add one. Given $\kappa$, $\lambda$, and $\mu$ as in the above definition, define a poset $\bb{B}(\kappa, \lambda, \mu)$ as follows. Conditions are of the form $p = \langle \mathcal{C}^p_\alpha \mid \alpha \in s^p \rangle$ such that:
\begin{itemize}
	\item{$s^p$ is a bounded subset of $\mu^+$ with a maximal element, which we denote $\gamma^p$;}
	\item{$\gamma^p \cap \mathrm{cof}(\geq \kappa) \subseteq s^p$;}
	\item{for all $\alpha \in s^p$ and all $C \in \mathcal{C}^p_\alpha$, $C$ is a club in $\alpha$ and $\mathrm{otp}(C) \leq \mu$;}
	\item{for all $\alpha \in s^p$, $1 \leq |\mathcal{C}^p_\alpha| < \lambda$;}
	\item{for all $\beta \in s^p$, for all $C \in \mathcal{C}^p_\beta$, and for all $\alpha \in C'$, we have $\alpha \in s^p$ and $C \cap \alpha \in \mathcal{C}_\alpha.$}
\end{itemize}

If $p,q \in \bb{B}(\kappa, \lambda, \mu)$, then $q \leq p$ iff $s^q$ end-extends $s^p$ and, for all $\alpha \in s^p$, $\mathcal{C}^q_\alpha = \mathcal{C}^p_\alpha$. 

The following is easily verified. See, for example, \cite{ac} for the details in the case $\lambda = 2$. The proof is essentially the same for other values of $\lambda$.

\begin{proposition}
	Let $\kappa$, $\lambda$, and $\mu$ be cardinals as above.
	\begin{enumerate}
		\item{$\bb{B}(\kappa, \lambda, \mu)$ is $\kappa$-directed closed.}
		\item{$\bb{B}(\kappa, \lambda, \mu)$ is $\mu+1$-strategically closed.}
		\item{$\Vdash_{\bb{B}(\kappa, \lambda, \mu)}``\square^{\geq \kappa}_{\mu, < \lambda}$ holds$."$}
	\end{enumerate}
\end{proposition}

\begin{proposition}
	Suppose $\kappa < \mu$ are infinite cardinals. Then the following are equivalent.
	\begin{enumerate}
		\item{$\square^{\geq \kappa^+}_{\mu, \mu}$ holds.}
		\item{There is a poset $\bb{P}$ such that $|\bb{P}| \leq \kappa$ and $\Vdash_{\bb{P}}``\square^*_\mu$ holds."}
	\end{enumerate}
\end{proposition}

\begin{proof}
	(1) $\Rightarrow$ (2): Suppose $\vec{\mathcal{C}} = \langle \mathcal{C}_\alpha \mid \alpha \in S \rangle$ is a $\square^{\geq \kappa^+}_{\mu, \mu}$-sequence. Let $\bb{P} = \mathrm{Coll}(\omega, \kappa)$. Then $|\bb{P}| = \kappa$ and, in $V^{\bb{P}}$, $\kappa^+ = \omega_1$. Define a $\square^*_\mu$-sequence $\mathcal{D} = \langle \mathcal{D}_\alpha \mid \alpha < \mu^+ \rangle$ in $V^{\bb{P}}$ as follows. If $\alpha \in S$, then let $\mathcal{D}_\alpha = \mathcal{C}_\alpha$. If $\alpha \in \lim(\mu^+) \setminus S$, then $(\cf(\alpha))^{V^{\bb{P}}} = \omega$. Let $D$ be an arbitrary $\omega$-sequence cofinal in $\alpha$, and let $\mathcal{D}_\alpha = \{D\}$. It is easily verified that this defines a $\square^*_\mu$-sequence.

	(2) $\Rightarrow$ (1): Suppose $\bb{P}$ is a forcing poset such that $|\bb{P}| \leq \kappa$ and $\Vdash_{\bb{P}} ``\square^*_\mu$ holds." Let $\dot{\vec{\mathcal{C}}} = \langle \dot{\mathcal{C}}_\alpha \mid \alpha < \mu^+ \rangle$ be a $\bb{P}$-name for a $\square^*_\mu$-sequence. For each $\bb{P}$-name $\dot{X}$ for a subset of $\mu^+$ and each $p \in \bb{P}$, let $\dot{X}_p = \{\alpha < \mu^+ \mid p \Vdash ``\alpha \in \dot{X}\}$. For each $\beta < \mu^+$, let $A_\beta = \{p \in \bb{P} \mid$ for some $\bb{P}$-name $\dot{C}$, $p \Vdash ``\dot{C} \in \dot{\mathcal{C}}_\beta"$ and $\dot{C}_p$ is club in $\beta \}$. Let $S = \{\beta < \mu^+ \mid A_\beta \not= \emptyset \}$. Easily, as $|\bb{P}| = \kappa$, $S^{\mu^+}_{\geq \kappa^+} \subseteq S$. Define a $\square^{\geq \kappa^+}_{\mu, \mu}$ sequence $\vec{\mathcal{D}} = \langle \mathcal{D}_\alpha \mid \alpha \in S \rangle$ by letting $\mathcal{D}_\alpha = \{\dot{C}_p \mid p \in A_\alpha, p \Vdash ``\dot{C} \in \dot{\mathcal{C}}_\alpha"$, and $\dot{C}_p$ is club in $\alpha\}$. It is easily verified that this is a $\square^{\geq \kappa^+}_{\mu, \mu}$-sequence.
\end{proof}

Since $\square^*_\mu$ is equivalent to the existence of a special $\mu^+$-Aronszajn tree, if $\square^{\geq \kappa^+}_{\mu, \mu}$ holds, then there is a strong $\mu^+$-system with $\kappa$ relations and no cofinal branch. We also get the following characterization for robustness of having no special trees.

\begin{corollary}
	Let $\mu$ be an infinite cardinal. The following are equivalent.
	\begin{enumerate}
		\item{There are no special $\mu^+$-Aronszajn trees, and this property holds robustly.}
		\item{$\square^{\geq \kappa^+}_{\mu, \mu}$ fails for all $\kappa < \mu$.}
	\end{enumerate}
\end{corollary}

\section{The strong system property at $\aleph_{\omega^2 + 1}$} \label{downsection}

In this section, we show how to obtain the strong system property at $\aleph_{\omega^2 + 1}$. We first recall the diagonal Prikry forcing introduced in \cite{free}.

Let $\langle \kappa_n \mid n < \omega \rangle$ be an increasing sequence of supercompact cardinals. Let $\mu = \sup(\{\kappa_n \mid n < \omega\})$, and let $\lambda = \mu^+$. Assume that, for all $n < \omega$, the supercompactness of $\kappa_n$ is indestructible under $\kappa_n$-directed closed forcing and $2^{\kappa_n} = \kappa_n^+$. For $m < \omega$, let $\bb{S}_m$ be the full-support product $\prod_{m \leq n < \omega} \mathrm{Coll}(\kappa_n^{+2}, < \kappa_{n+1})$. In $V^{\bb{S}_m}$, $\kappa_m$ remains supercompact, so we can fix in $V^{\bb{S}_m}$ a normal ultrafilter $F_m$ on $\mathcal{P}_{\kappa_m}(\lambda)$. Let $U_m$ be the projection of $F_m$ on $\kappa_m$, i.e., if $X \subseteq \kappa_m$, then $X \in U_m$ iff $\{y \in \mathcal{P}_{\kappa_m}(\lambda) \mid y \cap \kappa_m \in X\} \in F_m$. $U_m$ is then a normal ultrafilter on $\kappa_m$ and, since $\bb{S}_m$ is $\kappa^{+2}_m$-closed, $U_m \in V$. By the homogeneity of $\bb{S}_m$, we have that $U_m$ is forced by the empty condition in $\bb{S}_m$ to be the projection of a normal ultrafilter on $\mathcal{P}_{\kappa_m}(\lambda)$ in $V^{\bb{S}_m}$.

For $m < \omega$, let $M_m$ denote the transitive collapse of $\mathrm{Ult}(V, U_m)$, and let $j_m:V \rightarrow M_m$ be the associated embedding. Let $\bb{T}_m$ denote $\mathrm{Coll}(\kappa_m^{+\omega + 2}, < j_m(\kappa_m))$ as defined in $M_m$. $M \models ``$there are $j_m(\kappa_m)$ maximal antichains of $\bb{T}_m"$, $|j_m(\kappa_m)| = \kappa_m^+$, and $\bb{T}_m$ is $\kappa_m^+$-closed, so we can build in $V$ a filter $G_m$ that is $\bb{T}_m$-generic over $M_m$.

We now define the diagonal Prikry forcing, which we denote $\bb{P}$, from \cite{free}. For convenience, we let $\kappa_{-1} = \omega$. Conditions of $\bb{P}$ are of the form $p = \langle \alpha^p_0, \ldots, \alpha^p_{n-1}, \langle A^p_k \mid n \leq k < \omega \rangle, g^p_0, \ldots , g^p_n, f^p_0, \ldots, f^p_{n-1}, \langle F^p_k \mid n \leq k < \omega \rangle, \langle g^p_k \mid n < k < \omega \rangle \rangle$ such that:
\begin{itemize}
	\item{for all $i < n$, $\alpha^p_i$ is inaccessible and $\kappa_{i-1} < \alpha^p_i < \kappa_i$;}
	\item{for all $n \leq k < \omega$, $A^p_k \in U_k$ and, for all $\alpha \in A^p_k$, $\alpha$ is inaccessible;}
	\item{for all $i < n$, $g^p_i \in \mathrm{Coll}(\kappa_{i-1}^{+2}, < \alpha^p_i)$ and $f^p_i \in \mathrm{Coll}((\alpha_i^p)^{+\omega + 2}, < \kappa_i)$;}
	\item{for all $n \leq k < \omega$, $g^p_k \in \mathrm{Coll}(\kappa_{k-1}^{+2}, < \kappa_{k})$ and, for all $\alpha \in A^p_k$, $g^p_k \in \mathrm{Coll}(\kappa_{k+1}^{+2}, < \alpha)$;}
	\item{for all $n \leq k < \omega$, $F^p_k$ is a function with domain $A^p_k$ such that, for all $\alpha \in A^p_k$, $F^p_k(\alpha) \in \mathrm{Coll}(\alpha^{+\omega+2}, < \kappa_k)$ and $j_k(F^p_k)(\kappa_k) \in G_k$.}
\end{itemize}
$n$ is the length of $p$, denoted $\ell(p)$. If $q,p \in \bb{P}$, then $q \leq p$ iff:
\begin{itemize}
	\item{$\ell(q) \geq \ell(p)$;}
	\item{for all $i < \ell(p)$, $\alpha^q_i = \alpha^p_i$ and $f^q_i \leq f^p_i$;}
	\item{for all $i < \omega$, $g^q_i \leq g^p_i$;}
	\item{for all $\ell(q) \leq k < \omega$, $A^q_k \subseteq A^p_k$ and, for all $\alpha \in A^q_k$, $F^q_k(\alpha) \leq F^p_k(\alpha)$;}
	\item{for all $\ell(p) \leq k < \ell(q)$, $\alpha^q_k \in A^p_k$ and $f^q_k \leq F^p_k(\alpha^q_k)$.}
\end{itemize}

Following \cite{free}, given $p \in \bb{P}$ as above, we call $\langle \alpha^p_k \mid k < \ell(p) \rangle$ its $\alpha$-part, $\langle A^p_k \mid \ell(p) \leq k < \omega \rangle$ its $A$-part, $\langle f^p_k \mid k < \ell(p) \rangle$ its $f$-part, $\langle g^p_k \mid k \leq \ell(p) \rangle$ its $g$-part, $\langle F^p_k \mid \ell(p) \leq k < \omega \rangle$ its $F$-part, and $\langle g^p_k \mid \ell(p) < k < \omega \rangle$ its $S$-part. The $\alpha$-part, $g$-part, and $f$-part together comprise the lower part of $p$, denoted $a(p)$. If $k \leq \ell(p)$, let $p \restriction k$ denote $\langle \langle \alpha^p_i \mid i < k \rangle, \langle g^p_i \mid i \leq k \rangle, \langle f^p_i \mid i < k \rangle \rangle$. If $k > \ell(p)$, let $p \restriction k = a(p) ^\frown \langle A^p_i, F^p_i, g^p_{i+1} \mid \ell(p) \leq i < k \rangle$. Note that $p \restriction \ell(p) = a(p)$. We say that $q$ is a length-preserving extension of $p$ if $q \leq p$ and $\ell(q) = \ell(p)$. If $k \leq \ell(p)$, we say $q$ is a $k$-length-preserving extension of $p$ if $q$ is a length-preserving extension of $p$ and $q \restriction k = p \restriction k$. Finally, we say $q$ is a trivial extension of $p$ if it is an $\ell(p)$-length preserving extension of $p$.

The following facts hold about $\bb{P}$. Proofs can be found in \cite{free}.

\begin{itemize}
	\item{(\emph{Prikry property}) If $p \in \bb{P}$, $k \leq \ell(p)$, and $D$ is a dense open subset of $\bb{P}$, then there is a $k$-length preserving extension $q$ of $p$ such that, if $q^* \leq q$ and $q^* \in D$, then, if $q^{**} \leq q$, $\ell(q^{**}) = \ell(q^*)$, and $q^{**} \restriction k = q^* \restriction k$, then $q^{**} \in D$.}
	\item{$\bb{P}$ preserves all cardinals $\geq \mu$.}
	\item{The only cardinals below $\mu$ that are collapsed by forcing with $\bb{P}$ are those explicitly in the scope of the interleaved Levy collapses. In particular, if, in $V^{\bb{P}}$, $\langle \alpha_n \mid n < \omega \rangle$ is the generic Prikry sequence, then the cardinals below $\mu$ in $V^{\bb{P}}$ are precisely those in the intervals $\{[\kappa_{n-1}, \kappa_{n-1}^{+2}], [\alpha_n, \alpha_n^{+\omega+2}] \mid n < \omega \}$. It follows that, in $V^{\bb{P}}$, $\mu = \aleph_{\omega^2}$ and $\lambda = \aleph_{\omega^2 + 1}$.}
\end{itemize}

We now introduce an equivalence relation on $\bb{S}_0 = \prod_{n<\omega} \mathrm{Coll}(\kappa_n^{+2}, < \kappa_{n+1})$. Given $s = \langle s_n \mid n < \omega \rangle$ and $t = \langle t_n \mid n < \omega \rangle$ in $\bb{S}_0$, let $s \sim t$ if there is $m < \omega$ such that, for all $m \leq n < \omega$, $s_n = t_n$. Given $s \in \bb{S}_0$, let $[s]$ denote the equivalence class of $s$. If $m > 0$ and $s \in \bb{S}_m$, we abuse notation and let $[s]$ denote $[s^*]$, where $s^* \in \bb{S}_0$ is such that $s^*(n) = \emptyset$ if $n < m$ and $s^*(n) = s(n)$ if $n \geq m$. Let $\bb{S}^*$ be the poset whose conditions are equivalence classes $[s]$, where $s \in \bb{S}_0$. If $[s]$, $[t] \in \bb{S}^*$, we let $[t] \leq [s]$ iff there is $m < \omega$ such that, for all $m \leq n < \omega$, $t_n \leq s_n$. It is clear that this is well-defined.

If $p \in \bb{P}$, let the $S$-part of $p$ be denoted by $S(p)$. It is easily verified that, if $m < \omega$, the map $\pi_m:\bb{S}_m \rightarrow \bb{S}^*$ defined by $\pi_m(s) = [s]$ is a projection. Also, the map $\pi: \bb{P} \rightarrow \bb{S}^*$ defined by $\pi(p) = [S(p)]$ is a projection.

Let $G^*$ be $\bb{S}^*$-generic over $V$, and let $\bb{P}^* = \{p \in \bb{P} \mid [S(p)] \in G^*\}$. The following lemma was proved by Fontanella and Magidor in \cite{fontanella}. The proof there is for $m = 0$, but the general case is the same.

\begin{lemma} \label{cclemma}
	In $V[G^*]$, $\bb{P}^* \times \bb{S}_m / G^* \times \bb{S}_m/G^*$ has the $\lambda$-c.c.
\end{lemma}

\begin{theorem} \label{strongsystemthm}
	In $V^{\bb{P}}$, the strong system property holds at $\aleph_{\omega^2 + 1}$.
\end{theorem}

\begin{proof}
	We modify the proof of Theorem 1.2 in \cite{fontanella}. Recall that, in $V^{\bb{P}}$, $\mu = \aleph_{\omega^2}$ and $\lambda = \aleph_{\omega^2 + 1}$. Suppose the theorem fails. Then there is $p \in \bb{P}$ and a $\bb{P}$-name $\dot{S}$ forced by $p$ to be a strong $\lambda$-system with fewer than $\mu$ relations and no cofinal branch. We may assume $p \Vdash ``\mathrm{width}(\dot{S}) = \mu,"$ as other cases are easier. By strengthening $p$ if necessary, we may also assume the following.
	\begin{itemize}
		\item{There is $\nu < \mu$ such that $p$ forces $\dot{S}$ to be of the form $\langle \lambda \times \mu, \{\dot{R}_\eta \mid \eta < \nu\} \rangle$.}
		\item{$\nu < \kappa_{\ell(p)}$.}
	\end{itemize}
	 Let $m = \ell(p)$, and let $G$ be $\bb{S}_m$-generic over $V$ with $S(p) \in G$. Let $G^*$ be the $\bb{S}^*$-generic filter induced by $G$. Let $\bb{P}' = \{q \leq p \mid S(q) \in G\}$. We are abusing notation here in the sense that, if $q \in \bb{P}$ and $\ell(q) > m$, then $S(q)$ is not in $\bb{S}_m$. However, $G$ naturally projects to a generic $\bb{S}_{\ell(q)}$-generic filter $G_{\ell(q)}$, so $S(q) \in G$ should be interpreted as $S(q) \in G_{\ell(q)}$. The following is proven in \cite{free}.

	\begin{lemma}
		Forcing with $\bb{P}'$ over $V[G]$ adds a $\bb{P}$-generic filter over $V$.
	\end{lemma}

	Note also that, if $q_0, q_1 \in \bb{P}'$ and $a(q_0) = a(q_1)$, then $q_0$ and $q_1$ are compatible. In particular, $\bb{P}'$ has the $\lambda$-c.c. We now work in $V[G]$ and use the name $\dot{S}$, reinterpreted as a $\bb{P}'$-name, to extract a narrow system.

	\begin{lemma} \label{lengthlemma}
		There are $n,k < \omega$, $\eta < \nu$, and a cofinal set $I \subseteq \lambda$ such that, for all $\beta_0 < \beta_1$, both in $I$, there are $\gamma_0, \gamma_1 < \kappa_n$ and a condition $q \in \bb{P}'$ with $\ell(q) = k$ such that $q \Vdash ``(\beta_0, \gamma_0) <_{\dot{R}_\eta} (\beta_1, \gamma_1)"$.
	\end{lemma}

	\begin{proof}
		Recall that, in $V[G]$, $F_m$ is a normal measure on $\mathcal{P}_{\kappa_m}(\lambda)$ that projects to $U_m$. Let $M \cong \mathrm{Ult}(V, F_m)$ be the transitive collapse of the ultrapower, and let $j:V[G] \rightarrow M$ be the associated embedding. Find $r \leq j(p)$ in $j(\bb{P}')$ such that $\alpha^r_m = \kappa_m$. This is possible because, since $A^p_m \in U_m$, we have $\kappa_m \in j(A^p_m)$. Let $H$ be $j(\bb{P}')$-generic with $r \in H$. Note that, in $M[H]$, all cardinals in the interval $[\kappa_m, \kappa_m^{+\omega+2}]^{V[G]}$ are preserved. In particular, since $\lambda = (\kappa_m^{+\omega+1})^{V[G]}$, $\lambda$ is preserved in $M[H]$. Let $S^*$ be the realization of $j(\dot{S})$ in $M[H]$. Since $j(p) \in H$, $S^* = \langle j(\lambda) \times j(\mu), \{R^*_\eta \mid \eta < \nu\} \rangle$, where $R^*_\eta$ denotes the realization of $j(\dot{R}_\eta)$ for all $\eta < \nu$, is a strong $j(\lambda)$-system.

		Let $\delta = \sup(j``\lambda)$. For each $\beta < \lambda$, find $q_\beta \in H$, $\gamma^*_\beta < j(\mu)$, and $\eta_\beta < \nu$ such that $q_\beta \Vdash ``(j(\beta), \gamma^*_\beta) <_{j(\dot{R}_{\eta_\beta})} (\delta, 0)."$ There is then an unbounded $I^* \subseteq \lambda$, $n,k < \omega$, and $\eta < \nu$ such that, for all $\beta \in I^*$, we have $\ell(q_\beta) = k$, $\gamma^*_\beta < j(\kappa_n)$, and $\eta_\beta = \eta$.

		For $\beta \in I^*$, let $q_\beta = \langle \alpha^\beta_0, \ldots, \alpha^\beta_{k-1}, \langle A^\beta_i \mid k \leq i < \omega \rangle, g^\beta_0, \ldots, g^\beta_k, f^\beta_0, \ldots, f^\beta_{k-1}, \langle F^\beta_i \mid k \leq i < \omega \rangle, \langle g^\beta_i \mid k < i < \omega \rangle \rangle$. Since the $q_\beta$'s are pairwise compatible, there is a sequence $\langle \alpha_i \mid i < k \rangle$ such that, for all $\beta \in I^*$ and all $i < k$, $\alpha^\beta_i = \alpha_i$. Moreover, we know that $k > m$, $\alpha_i = \alpha^p_i$ for all $i < m$, and $\alpha_m = \kappa_m$. There are thus fewer than $\lambda$-many choices for the sequence $\langle g^\beta_0, \ldots, g^\beta_m, f^\beta_0, \ldots, f^\beta_{m-1} \rangle$, so we can find a cofinal $I \subseteq I^*$ and a sequence $\langle g_0, \ldots, g_m, f_0, \ldots, f_{m-1} \rangle$ such that, for all $\beta \in I$, $\langle g^\beta_0, \ldots, g^\beta_m, f^\beta_0, \ldots, f^\beta_{m-1} \rangle = \langle g_0, \ldots, g_m, f_0, \ldots, f_{m-1} \rangle$. Finally, if $m < i < k$ and $\beta \in I$, then $g_i$ comes from a forcing that is $\lambda^+$-directed closed. Similarly, if $m \leq i < k$ and $\beta \in I$, then $f_i$ comes from a forcing that is $\lambda^+$-directed closed. Thus, since $M$ is closed under $\lambda$-sequences, we may assume, by taking lower bounds on the relevant coordinates, that there is a lower part $a^* = \langle \alpha_0, \ldots, \alpha_{k-1}, g_0, \ldots, g_k, f_0, \ldots, f_{k-1} \rangle$ such that, for all $\beta \in I$, $a(q_\beta) = a^*$.

		We claim that $I$, $n$, $k$, and $\eta$ are as desired. Work in $V[G]$. Fix $\beta_0 < \beta_1$, both in $I$. In $M$, we have $q_{\beta_0}, q_{\beta_1} \in j(\bb{P}')$ with $a(q_{\beta_0}) = a(q_{\beta_1}) = a^*$ and $\gamma^*_{\beta_0}, \gamma^*_{\beta_1} < j(\kappa_n)$ such that, for $\epsilon < 2$, $q_\epsilon \Vdash ``(j(\beta_\epsilon), \gamma^*_{\beta_\epsilon}) <_{j(\dot{R}_\eta)} (\delta, 0)."$ Since $a(q_{\beta_0}) = a(q_{\beta_1}) = a^*$, we can find $q^* \leq q_{\beta_0}, q_{\beta_1}$ with $a(q^*) = a^*$. Then, since $j(\dot{R}_\eta)$ is forced to be tree-like, we have $q^* \Vdash ``(j(\beta_0), \gamma^*_{\beta_0}) <_{j(\dot{R}_\eta)} (j(\beta_1), \gamma^*_{\beta_1})."$ By elementarity, there are $q \in \bb{P}'$ with $\ell(q) = k$ and $\gamma_0, \gamma_1 < \kappa_n$ such that $q \Vdash ``(\beta_0, \gamma_0) <_{\dot{R}_\eta} (\beta_1, \gamma_1)."$
	\end{proof}
	Fix $n$, $k$, $\eta$, and $I$ as in Lemma \ref{lengthlemma}. Define a system $S_0 = \langle I \times \kappa_n, \{R_a \mid a$ is a lower part of length $k\} \rangle$ by letting $(\beta_0, \gamma_0) <_{R_a} (\beta_1, \gamma_1)$ if and only if there is $q \in \bb{P}'$ such that $a(q) = a$ and $q \Vdash ``(\beta_0, \gamma_0) <_{\dot{R}_\eta} (\beta_1, \gamma_1)."$ By Lemma \ref{lengthlemma}, $S_0$ is a narrow $\lambda$-system. By Lemma \ref{collapselemma}, $\lambda$ satisfies the narrow system property in $V[G]$, so there is a cofinal branch through $S_0$. Namely, there is a cofinal $J \subseteq I$ and a lower part $a$ of length $n$ such that, for every $\beta \in J$, there is $\gamma_\beta < \kappa_n$ such that, whenever $\beta_0 < \beta_1$ are both in $J$, there is $q \in \bb{P}'$ with $a(q) = a$ such that $q \Vdash ``(\beta_0, \gamma_{\beta_0}) <_{\dot{R}_\eta} (\beta_1, \gamma_{\beta_1})"$. Fix such a $J$ and $a$ and an assignment $\beta \mapsto \gamma_\beta$ for $\beta \in J$.

	For $k \leq i < \omega$, let $H_i$ be the set of $(A,F,g)$ such that $A \in U_i$, $F$ is a function with domain $A$ such that, for all $\alpha \in A$, $F(\alpha) \in \mathrm{Coll}(\alpha^{+\omega+2}, < \kappa_i)$ and $j_i(F)(\kappa_i) \in G_i$, and $g \in \mathrm{Coll}(\kappa^{+2}_i, < \kappa_{i+1})$ is in the generic filter induced by $G$. Suppose that, for $\epsilon < 2$, $(A^\epsilon_i, F^\epsilon_i, g^\epsilon_{i+1}) \in H_i$. Then we define $(A^0_i, F^0_i, g^0_{i+1}) \wedge (A^1_i, F^1_i, g^1_{i+1}) = (A_i,F_i,g_{i+1})$ by letting $A_i = \{\alpha \in A^0_i \cap A^1_i \mid F^0_i(\alpha)$ and $F^1_i(\alpha)$ are compatible$\}$, defining $F_i$ on $A_i$ by $F_i(\alpha) = F^0_i(\alpha) \cup F^1_i(\alpha)$, and letting $g_{i+1} = g^0_{i+1} \cup g^1_{i+1}$. Suppose that, for $\epsilon < 2$, $q_\epsilon \in \bb{P}'$ and $a(q_\epsilon) = a$, where $q_\epsilon = a ^\frown \langle A^\epsilon_i, F^\epsilon_i, g^\epsilon_{i+1} \mid k \leq i < \omega \rangle$. Then the greatest lower bound for $q_0$ and $q_1$ is $q = a ^\frown \langle A_i, F_i, g_{i+1} \mid k \leq i < \omega \rangle$ where, for all $k \leq i < \omega$, $(A_i, F_i, g_{i+1}) = (A^0_i, F^0_i, g^0_{i+1}) \wedge (A^1_i, F^1_i, g^1_{i+1})$. 

	By recursion on $i$, we now define $\langle (A^\beta_i, F^\beta_i, g^\beta_{i+1}) \mid \beta \in J, k \leq i < \omega \rangle$, maintaining the inductive hypothesis that, for all $\beta_0 < \beta_1$, both in $J$, and all $\ell$ with $k \leq \ell < \omega$, there is $q \in \bb{P}'$ such that $q \restriction \ell = a^{\frown} \langle (A^{\beta_0}_i, F^{\beta_0}_i, g^{\beta_0}_{i+1}) \wedge (A^{\beta_1}_i, F^{\beta_1}_i, g^{\beta_1}_{i+1}) \mid k \leq i < \ell \rangle$ and $q \Vdash ``(\beta_0, \gamma_{\beta_0}) <_{\dot{R}_\eta} (\beta_1, \gamma_{\beta_1})."$

	Suppose $k \leq \ell < \omega$ and $(A^\beta_i, F^\beta_i, g^\beta_{i+1})$ has been defined for all $\beta \in J$ and all $k \leq i < \ell$. Define a system $S_\ell = \langle J \times \{0\}, \{R_{A,F,g}\mid (A,F,g) \in H_\ell\}$ by letting $(\beta_0, 0) <_{R_{A,F,g}} (\beta_1, 0)$ iff there is $q \in \bb{P}'$ such that $q \restriction (j+1) = a^\frown \langle(A^{\beta_0}_i, F^{\beta_0}_i, g^{\beta_0}_{i+1}) \wedge (A^{\beta_1}_i, F^{\beta_1}_i, g^{\beta_1}_{i+1}) \mid k \leq i < \ell \rangle ^\frown \langle (A,F,g) \rangle$ and $q \Vdash ``(\beta_0, \gamma_{\beta_0}) <_{\dot{R}_\eta} (\beta_1, \gamma_{\beta_1})."$ By the inductive hypothesis, this defines a narrow $\lambda$-system, so, again by Lemma \ref{collapselemma}, it has a cofinal branch, namely a cofinal set $J_\ell \subseteq J$ and a fixed $(A,F,g) \in H_j$ such that, for all $\beta_0 < \beta_1$, both in $J_\ell$, $(\beta_0, 0) <_{R_{A,F,g}} (\beta_1, 0)$. We now define $(A^\beta_\ell, F^\beta_\ell, g^\beta_{\ell+1})$ for $\beta \in J$ as follows. If $\beta \in J_\ell$, then $(A^\beta_\ell, F^\beta_\ell, g^\beta_{\ell+1}) = (A,F,g)$. If $\beta \not\in J_\ell$, let $\beta^* = \min(J_\ell \setminus \beta)$. Find $q \in \bb{P}'$ such that $q \restriction \ell = a^{\frown} \langle(A^\beta_i, F^\beta_i, g^\beta_{i+1}) \wedge (A^{\beta^*}_i, F^{\beta^*}_i, g^{\beta^*}_{i+1}) \mid k \leq i < \ell \rangle$ and $q \Vdash ``(\beta, \gamma_\beta) <_{\dot{R}_\eta} (\beta^*, \gamma_{\beta^*})"$, and let $(A^\beta_\ell, F^\beta_\ell, g^\beta_{\ell+1}) = (A^q_\ell, F^q_\ell, g^q_{\ell+1}) \wedge (A,F,g)$. It is tedious but straightforward to verify that this definition maintains the inductive hypothesis.

	For $\beta \in J$, let $q_\beta = a^\frown \langle (A^\beta_i, F^\beta_i, g^\beta_{i+1}) \mid k \leq i < \omega \rangle$ and note that $q_\beta \in \bb{P}'$. For $\beta_0 < \beta_1$, both in $J$, let $q_{\beta_0, \beta_1} := q_{\beta_0} \wedge q_{\beta_1}$ denote the greatest lower bound of $q_{\beta_0}$ and $q_{\beta_1}$.

	\begin{claim}
		Suppose $\beta_0 < \beta_1$, both in $J$. Then $q_{\beta_0, \beta_1} \Vdash ``(\beta_0, \gamma_{\beta_0}) <_{\dot{R}_\eta} (\beta_1, \gamma_{\beta_1})."$
	\end{claim}

	\begin{proof}
		Suppose not, and let $r \in \bb{P}'$, $r \leq q_{\beta_0, \beta_1}$ be such that $r \Vdash ``(\beta_0, \gamma_{\beta_0}) \not<_{\dot{R}_\eta} (\beta_1, \gamma_{\beta_1})."$ Let $i = \ell(r)$. By the inductive hypothesis in the previous construction, we can find $q \in \bb{P'}$ such that $q \restriction i = q_{\beta_0, \beta_1} \restriction i$ and $q \Vdash ``(\beta_0, \gamma_{\beta_0}) <_{\dot{R}_\eta} (\beta_1, \gamma_{\beta_1})"$. But it is easily seen that $r$ and $q$ are compatible in $\bb{P}'$, which is a contradiction.
	\end{proof}

	\begin{claim}
		There is $q \in \bb{P}'$ such that $q \Vdash ``$for unboundedly many $\beta \in J, q_\beta \in \dot{H}"$, where $\dot{H}$ is the canonical name for the generic filter.
	\end{claim}

	\begin{proof}
		Suppose not. Then there is a maximal antichain $A \subseteq \bb{P}'$ such that, for every $r \in A$, there is $\beta_r < \lambda$ such that, for all $\beta \in J \setminus \beta_r$, $r \Vdash ``q_\beta \not\in \dot{H}."$ Since $\bb{P}'$ has the $\lambda$-c.c., $|A| < \lambda$, so $\beta^* := \sup(\{\beta_r \mid r \in A\}) < \lambda$, and therefore, for all $\beta \in J \setminus \beta^*$, $\Vdash ``q_\beta \not\in \dot{H}."$ But clearly, if $\beta \in J \setminus \beta^*$, $q_\beta \Vdash ``q_\beta \in \dot{H},"$ which is a contradiction.
	\end{proof}

	Let $q \in \bb{P}'$ be as given in the previous claim, and let $H$ be $\bb{P}'$-generic with $q \in H$. Let $S$ be the realization of $\dot{S}$. Then, in $V[G][H]$, $\{(\beta, \gamma_\beta) \mid \beta \in J$ and $q_\beta \in H\}$ is a cofinal branch of $S$ through $R_\eta$. Thus, $S$ has a cofinal branch in $V[G][H]$. Note that, as $H$ is $\bb{P}'$-generic over $V[G]$, it is also $\bb{P}^*$-generic over $V[G^*]$, so $V[G][H] = V[G^*][H][G/G^*] = V[H][G/G^*]$, where $G/G^*$ is $\bb{S}_m/G^*$-generic over $V[G^*][H] = V[H]$. By Lemma \ref{cclemma}, $\bb{S}_m/G^* \times \bb{S}_m/G^*$ has the $\lambda$-c.c. in $V[H]$, so, by Lemma \ref{approxlemma}, $\bb{S}_m/G^*$ has the $\lambda$-approximation property in $V[H]$. Therefore, by Lemma \ref{approxbranchprop}, $S$ has a cofinal branch in $V[H]$. But $p \in H$ and $p$ forced that $\dot{S}$ had no cofinal branch. This is a contradiction.
\end{proof}

\section{Separating properties} \label{separatingSect}

In this section, we show that the tree property and the strong system property are in general not equivalent at $\aleph_{\omega^2 + 1}$ and that we have some amount of control over the number of relations needed to construct a strong $\aleph_{\omega^2+1}$-system with no cofinal branch. We first review some facts about forcing and square sequences.

Recall that, if $\kappa, \lambda$, and $\mu$ are cardinals, with $\kappa \leq \mu$, $\kappa$ regular, and $\lambda > 1$, then $\bb{B}(\kappa, \lambda, \mu)$ is the forcing to add a $\square^{\geq \kappa}_{\mu, < \lambda}$-sequence. For our purposes, let $\kappa < \mu$ with $\kappa$ regular and $\mu$ singular of cofinality $\omega$, and let $\bb{B} = \bb{B}(\kappa, \mu, \mu)$.

In $V^{\bb{B}}$, let $\vec{\mathcal{C}} = \langle \mathcal{C}_\alpha \mid \alpha \in S \rangle$ be the $\square^{\geq \kappa}_{\mu, < \mu}$-sequence added by $\bb{B}$. Let $\nu < \mu$ be regular, and let $\bb{T}_\nu$ be the forcing poset whose conditions are closed, bounded subsets $t$ of $\mu^+$ such that:
\begin{itemize}
	\item{$|t| < \nu$;}
	\item{for all $\alpha \in t'$, $\alpha \in S$ and $t \cap \alpha \in \mathcal{C}_\alpha$.}
\end{itemize}
If $s,t \in \bb{T}_\nu$, then $t \leq s$ iff $t$ end-extends $s$. For a cardinal $\epsilon$, let $\bb{T}_\nu^\epsilon$ denote the full-support product of $\epsilon$ copies of $\bb{T}_\nu$. Let $\dot{\bb{T}}_\nu$ be a $\bb{B}$-name for $\bb{T}_\nu$.

\begin{proposition}
	In $V$, for every $\epsilon < \mu$, $\bb{B} * \dot{\bb{T}}_\nu^\epsilon$ has a dense $\nu$-closed subset.
\end{proposition}

\begin{proof}
	Let $\bb{U}$ be the set of $(p, \langle \dot{t}_\eta \mid \eta < \epsilon \rangle) \in \bb{B} * \dot{\bb{T}}$ such that, for all $\eta < \epsilon$:
	\begin{itemize}
		\item{there is $t_\eta \in V$ such that $p \Vdash ``\dot{t}_\eta = t_\eta"$;}
		\item{$\gamma^p = \max(t_\eta)$.}
	\end{itemize}
	The verification that $\bb{U}$ is dense and $\nu$-closed is straightforward. See, for example, \cite{cfm} for similar arguments.
\end{proof}

\begin{corollary}
	Forcing with $\bb{T}_\nu$ over $V^{\bb{B}}$ adds a club $D$ in $\mu^+$ such that:
	\begin{itemize}
		\item{$\mathrm{otp}(D) = \nu$;}
		\item{for all $\alpha \in D'$, $\alpha \in S$ and $D \cap \alpha \in \mathcal{C}_\alpha$.}
	\end{itemize}
\end{corollary}

Let $\langle \kappa_n \mid n < \omega \rangle$ be an increasing sequence of indestructibly supercompact cardinals. Let $\mu = \sup(\{\kappa_n \mid n < \omega\})$, and let $\lambda = \mu^+$. As before, for $m < \omega$, let $\bb{S}_m$ be the full-support product $\prod_{m \leq n < \omega} \mathrm{Coll}(\kappa_n^{+2}, < \kappa_{n+1})$, and let $\bb{S}^*$ be defined as in the previous section. Fix an $m < \omega$, and, in $V^{\bb{S}^*}$, let $\bb{B} = \bb{B}(\kappa_m, \mu, \mu)$. Let $\dot{\bb{B}}$ be an $\bb{S}^*$-name for $\bb{B}$, and note that $\bb{S}_m * \dot{\bb{B}}$ is $\kappa_m$-directed closed.

\begin{lemma} \label{distributivelemma}
	In $V^{\bb{S}_m}$, $\bb{B}$ is $\lambda$-distributive.
\end{lemma}

\begin{proof}
	Since $\lambda = \mu^+$ and $\mu$ is singular, it suffices to show that $\bb{B}$ is $\nu$-distributive for all regular $\nu < \mu$. In particular, it suffices to show that $\bb{B}$ is $\kappa_n$-distributive for all $n < \omega$. To this end, fix $n < \omega$. Without loss of generality, $m \leq n$. In $V^{\bb{S^*} * \dot{\bb{B}}}$, let $\bb{T} = \bb{T}_{\kappa_n}$, as defined above. Recall that, in $V^{\bb{S^*}}$, $\bb{B} * \dot{\bb{T}}$ has a $\kappa_n$-closed dense subset. In $V$, let $\bb{S}_{m,n} = \prod_{m \leq k < n} \mathrm{Coll}(\kappa_k^{+2}, < \kappa_{k+1})$. Note that $\bb{S}_m \cong \bb{S}_{m, n} \times \bb{S}_n$ and that $\bb{S}_{m,n}$ has the $\kappa_n$-c.c. Since $\bb{S}_n$ is $\kappa_n^{+2}$-closed, we still have in $V^{\bb{S}_n}$ that $\bb{S}_{m,n}$ has the $\kappa_n$-c.c. and $\bb{B} * \dot{\bb{T}}$ has a $\kappa_n$-closed dense subset. Thus, by Easton's Lemma, $\bb{B} * \dot{\bb{T}}$ is $\kappa_n$-distributive in $V^{\bb{S}_n \times \bb{S}_{m,n}} = V^{\bb{S}_m}$. In particular, $\bb{B}$ is $\kappa_n$-distributive in $V^{\bb{S}_m}$. 
\end{proof}

\begin{lemma}
	In $V^{\bb{S}_m * \dot{\bb{B}}}$, $\lambda$ satisfies the narrow system property.
\end{lemma}

\begin{proof}
	First note that, by Lemma \ref{distributivelemma}, $\lambda = \kappa_{m}^{+\omega+1}$ in $V^{\bb{S}_{m} * \dot{\bb{B}}}$. Let $\dot{S}$ be an $\bb{S}_{m} * \dot{\bb{B}}$-name for a narrow $\lambda$-system. Without loss of generality, we may assume there are $\kappa, \nu < \mu$ such that $\dot{S}$ is forced to be of the form $\langle \lambda \times \kappa, \{\dot{R}_\eta \mid \eta < \nu\} \rangle$. Fix $m < n^* < \omega$ such that $\kappa, \nu < \kappa_{n^*}$. In $V^{\bb{S}^* * \dot{\bb{B}}}$, let $\bb{T} = \bb{T}_{\kappa_{n^*+1}}$. Let $G_{n^*+1}$ be $\bb{S}_{n^*+1}$-generic over $V$ and let $H*I$ be $\bb{B} * \dot{\bb{T}}$-generic over $V[G_{n^*+1}]$. Since $\bb{S}_{n^*+1} * \dot{\bb{B}} * \dot{\bb{T}}$ has a $\kappa_{n^*+1}$-closed dense subset, $\kappa_{n^*+1}$ remains supercompact in $V[G_{n^*+1}*H*I]$. Move to $V[G_{n^*+1}*H*I]$, which we denote $V_1$, reinterpreting $\dot{S}$ as an $\bb{S}_{m, n^*+1}$-name.

	Let $j:V_1 \rightarrow M$ witness that $\kappa_{n^*+1}$ is $\lambda$-supercompact. $j(\bb{S}_{m,n^*}) = \bb{S}_{m,n^*}$ and $j(\bb{S}_{n^*, n^*+1}) = \mathrm{Coll}(\kappa_{n^*}^{+2}, < j(\kappa_{n^*+1})) \cong \bb{S}_{n^*, n^*+1} * \dot{\bb{R}}$, where $\dot{\bb{R}}$ is $\kappa_{n^*}^{+2}$-closed in $V_1^{\bb{S}_{n^*, n^*+1}}$. Let $G_{n^*, n^*+1}$ be $\bb{S}_{n^*, n^*+1}$-generic over $V_1$, and let $\bb{R}$ be the interpretation of $\dot{\bb{R}}$ in $V_1[G_{n^*, n^*+1}]$. Then, letting $G_{m,n^*}$ be $\bb{S}_{m,n^*}$ over $V_1[G_{n^*, n^*+1}]$ and $J$ be $\bb{R}$-generic over $V_1[G_{n^*, n^*+1}][G_{m^*,n}]$, we can lift $j$ to $j:V_1[G_{n^*,n^*+1}][G_{m,n^*}] \rightarrow M[G_{n^*,n^*+1}][G_{m,n^*}][J]$.

	Let $S = \langle \lambda \times \kappa, \{R_\eta \mid \eta < \nu \} \rangle$ be the realization of $\dot{S}$ in $V_1[G_{n^*,n^*+1}][G_{m,n^*}]$. In $M[G_{n^*,n^*+1}][G_{m,n^*}][J]$, $j(S) = \langle j(\lambda) \times \kappa, \{j(R_\eta) \mid \eta < \nu\} \rangle$ is a $j(\lambda)$-system. Let $\delta = \sup(j``\lambda)$. For all $\gamma < \kappa$ and $\eta < \nu$, let $b_{\gamma, R_\eta} = \{(\alpha, \beta) \in \lambda \times \kappa \mid (j(\alpha), \beta) <_{j(R_\eta)} (\delta, \gamma)\}$. $\bar{b} := \{b_{\gamma, R_\eta} \mid \gamma < \kappa, \eta < \nu\}$ is a full set of branches through $S$, and $\bar{b} \in V[G_{n^*+1}*H*I][G_{n^*, n^*+1}][G_{m,n^*}][J]$.

	In $V[G_{n^* + 1}]$, $\bb{B} * \dot{\bb{T}}^{\kappa_{n^*}}$ has a dense $\kappa_{n^*+1}$-closed subset. Moreover, $\bb{S}_{n^*, n^*+1}$ is $\kappa_{n^*}$-closed in $V[G_{n^*+1}]$ and, in $V[G_{n^*+1}*H*I][G_{n^*,n^*+1}]$, $\bb{R}$ is $\kappa_{n^*}$-closed. Thus, in $V[G_{n^*+1}][G_{n^*, n^*+1}]$, $\bb{B} * \dot{\bb{T}}^{\kappa_{n^*}} * \dot{\bb{R}}^{\kappa_{n^*}} \cong \bb{B} * (\dot{\bb{T}} * \dot{\bb{R}})^{\kappa_{n^*}}$ has a dense $\kappa_{n^*}$-closed subset. $\bb{S}_{m,n^*}$ has the $\kappa_{n^*}$-c.c. in $V[G_{n^*+1}][G_{n^*, n^*+1}]$, so, by Easton's Lemma, $\bb{B} * (\dot{\bb{T}} * \dot{\bb{R}})^{\kappa_{n^*}}$ is $\kappa_{n^*}$-distributive in $V[G_{n^*+1}][G_{n^*, n^*+1}][G_{m,n^*}]$, so $(\bb{T} * \dot{\bb{R}})^{\kappa_{n^*}}$ is $\kappa_{n^*}$-distributive in $V[G_{n^*+1}][G_{n^*, n^*+1}][G_{m,n^*}][H]$. Thus, we can apply Lemma \ref{neemanlemma} in $V[G_{n^*+1}][G_{n^*, n^*+1}][G_{m,n^*}][H]$ to $\bb{T}*\dot{\bb{R}}$ to conclude that $S$ has a cofinal branch in $V[G_{n^*+1}][G_{n^*, n^*+1}][G_{m,n^*}][H]$, thereby completing the proof.
\end{proof}

For $k < \omega$, let $U_k$ be a normal ultrafilter on $\kappa_k$. For $k \not= m$, the choice is arbitrary. For $k = m$, note that $\kappa_m$ remains supercompact in $V^{\bb{S}_m * \dot{\bb{B}}}$. Choose a condition $(s^*, \dot{t}^*) \in \bb{S}_m * \dot{\bb{B}}$ and a normal ultrafilter $U_m$ on $\kappa_m$ such that $(s^*, \dot{t}^*)$ forces $U_m$ to be the projection of a normal ultrafilter $F_m$ on $\mathcal{P}_{\kappa_m}(\lambda)$ in $V^{\bb{S}_m * \dot{\bb{B}}}$. Using, the $U_k$'s, define filters $G_k$ and a diagonal Prikry poset $\bb{P}$ as in the previous section.

Recall that $\kappa_{-1} = \omega$. In $V^{\bb{P} * \dot{\bb{B}}}$, $\kappa_m = \aleph_{\omega \cdot (m + 1) + 3}$, $\mu = \aleph_{\omega^2}$, and $\lambda = \aleph_{\omega^2 + 1}$. Also, $\square^{\geq \kappa_m}_{\mu, < \mu}$ holds. In particular, there is a strong $\lambda$-system with $\aleph_{\omega \cdot (m + 1) + 2}$ conditions.

\begin{theorem}
	There is a generic extension by $\bb{P} * \dot{\bb{B}}$ in which every strong $\lambda$-system with $\kappa_{m-1}^{+2}$ relations has a cofinal branch.
\end{theorem}

\begin{proof}
	Suppose not. Then there is a $\bb{P} * \dot{\bb{B}}$-name $\dot{S} = \langle \lambda \times \mu, \{\dot{R}_\eta \mid \eta < \kappa_{m-1}^{+2}\} \rangle$ such that $\Vdash_{\bb{P} * \dot{\bb{B}}} ``\dot{S}$ is a strong $\lambda$-system with no cofinal branch."

	Let $G * B$ be $\bb{S}_m * \dot{\bb{B}}$-generic over $V$ with $(s^*, \dot{t}^*) \in G * B$. Let $p \in \bb{P}$ be such that $\ell(p) = m$ and $S(p) = s^*$. As in the previous section, let $G^*$ be the $\bb{S}^*$-generic filter induced by $G$, let $\bb{P}' = \{q \leq p \mid S(q) \in G\}$, and let $\bb{P}^* = \{q \leq p \mid [S(q)] \in G^*\}$.

	\begin{lemma} \label{cclemma2}
		In $V[G^* * B]$, $\bb{P}^* \times \bb{S}_m/G^* \times \bb{S}_m/G^*$ has the $\lambda$-c.c.
	\end{lemma}

	\begin{proof}
		An inspection of the proof of Lemma \ref{cclemma} (which is found as Lemma 4.5 in \cite{fontanella}) reveals that the same proof works when $\bb{S}^*$ is replaced by $\bb{S^*} * \dot{\bb{A}}$ for any $\bb{S^*}$-name $\dot{\bb{A}}$ for a $\mu + 1$-strategically closed forcing. Because $\dot{\bb{B}}$ is forced to be $\mu+1$-strategically closed, the proof goes through.
	\end{proof}

	Now work in $V[G * B]$ and repeat the proof of Theorem \ref{strongsystemthm}, beginning at Lemma \ref{lengthlemma} and using Lemma \ref{cclemma2} in place of Lemma \ref{cclemma}.
\end{proof}

It is clear that, in any extension by $\bb{P} * \dot{\bb{B}}$, $\square^{\geq \kappa_m}_{\mu, < \mu}$ holds. In particular, $\lambda$ fails to satisfy the robust tree property. Since our choice of $m < \omega$ was arbitrary, we have shown the following.

\begin{corollary}
	Suppose there are infinitely many supercompact cardinals, and let $\alpha < \omega^2$. Then there is a forcing extension in which every strong $\aleph_{\omega^2 + 1}$-system with $\aleph_\alpha$ conditions has a cofinal branch but in which $\aleph_{\omega^2 + 1}$ fails to have the robust tree property.
\end{corollary}

We end with three open questions.

\begin{question}
	Can $\aleph_{\omega + 1}$ consistently satisfy the strong system property?
\end{question}

\begin{question} \label{question2}
	Suppose $\lambda$ is a regular uncountable cardinal and $\lambda$ satisfies the robust tree property. Must $\lambda$ satisfy the strong system property?
\end{question}

\begin{question} \label{question3}
	Suppose $\lambda$ is a regular uncountable cardinal and $\lambda$ satisfies the tree property. Must it be the case that every strong $\lambda$-system with only countably many relations has a cofinal branch?
\end{question}

Note that a `Yes' answer to Question \ref{question3} would also entail a `Yes' answer to Question \ref{question2}.

\bibliography{strong_systems}
\bibliographystyle{plain}

\end{document}